\documentclass{amsart}
\usepackage{amsmath}
\usepackage{amsthm}
\usepackage{amssymb}
\usepackage{enumerate}
\usepackage[foot]{amsaddr}

\theoremstyle{plain}
\newtheorem{theorem}{Theorem}[section]
\newtheorem{proposition}[theorem]{Proposition}
\newtheorem{corollary}[theorem]{Corollary}
\newtheorem{lemma}[theorem]{Lemma}

\theoremstyle{definition}
\newtheorem{definition}[theorem]{Definition}
\newtheorem{notation}[theorem]{Notation}
\newtheorem{construction}[theorem]{Construction}

\theoremstyle{remark}
\newtheorem{remark}[theorem]{Remark}
\newtheorem{example}[theorem]{Example}

\newcommand{\ider}{{D}}

\newcommand{\fa}{\mathfrak{a}}
\newcommand{\fg}{\mathfrak{g}}
\newcommand{\fh}{\mathfrak{h}}
\newcommand{\fm}{\mathfrak{m}}

\numberwithin{equation}{section}

\title[Representations of $L$-semisimple Lie-Yamaguti algebras]
{Representations of Lie-Yamaguti algebras with \\
semisimple enveloping Lie algebras}

\author{Nobuyoshi Takahashi}
\address{
Department of Mathematics, 
Graduate School of Advanced Science and Engineering, 
Hiroshima University, 
1-3-1 Kagamiyama, Higashi-Hiroshima, 
739-8526 JAPAN}
\email{tkhsnbys@hiroshima-u.ac.jp}



\begin{document}

\begin{abstract}
Let $T$ be a Lie-Yamaguti algebra whose standard enveloping Lie algebra $L(T)$ is 
semisimple and $[T, T, T]=T$. 
Then we give a description of representations of $T$ 
in terms of representations of $L(T)$ with certain additional data. 
Similarly, if $(T, \sigma)$ is an infinitesimal $s$-manifold such that $L(T)$ is semisimple, 
then any representation of $(T, \sigma)$ 
comes from a representation of $L(T)$. 
\end{abstract}

\maketitle

\section{Introduction}

In his study of invariant affine connections on reductive homogeneous spaces, 
Nomizu(\cite{Nomizu1954}) considered the tangent space of a reductive homogeneous space at a point 
together with the torsion and curvature tensors 
associated to a certain canonical affine connection. 
It can be regarded as a kind of algebra given by a bilinear and a trilinear operations 
which generalizes both Lie algebras and Lie triple systems. 
Yamaguti called them general Lie triple systems, 
and more detailed studies were made by Yamaguti(\cite{Yamaguti1958}, \cite{Yamaguti1969}), 
Sagle(\cite{Sagle1965a}, \cite{Sagle1965b}, \cite{Sagle1968}) 
and Kikkawa(\cite{Kikkawa1979}, \cite{Kikkawa1981}). 
Such an algebra is now called a Lie-Yamaguti algebra (\cite{KW2001}, see Definition \ref{def_lyalg}), 
and is an object of active study. 
See \cite{BDE2005}, \cite{BEMH2009}, \cite{BEMH2011}, \cite{GMM2024}
for topics related to structure theories, 
\cite{SZZ2021}, \cite{ZXQ2023} for additional structures such as complex structures,  
Rota-Baxter operators and post-structures, 
and \cite{GN2012}, \cite{ZL2015}, \cite{ZL2017}, \cite{SC2023}, \cite{TJL2023} 
for deformations and Hom- and differential generalizations. 

Just as Lie triple systems are infinitesimal algebras for symmetric spaces, 
the study of Lie-Yamaguti algebras is closely related to that of regular $s$-manifolds. 
A regular $s$-manifold $Q$ is a manifold with a ``symmetry'' at each point, 
not necessarily involutive, 
satisfying certain conditions (\cite{Kowalski1980}, \cite{Fedenko1977}, see Definition \ref{def_reg_s_mfd}). 
The tangent space $T=T_qQ$ of a regular $s$-manifold $Q$ at a point $q$ 
can be endowed with the structure of a Lie-Yamaguti algebra, 
and the symmetry at the point gives an automorphism $\sigma$ of $T$. 
Then the Lie-Yamaguti algebra $T$ together with $\sigma$ forms 
what is called an infinitesimal $s$-manifold (see Definition \ref{def_inf_s_mfd}), 
and there is an inverse correspondence as in the case of Lie groups. 

A regular $s$-manifold can be considered as a quandle. 
A quandle is an algebraic system defined by a binary operation 
(\cite{Joyce1982}, \cite{Matveev1982}, see Definition \ref{def_quandle}), 
with many applications in knot theory. 
While current applications mainly make use of finite quandles, 
studies have recently been started on 
topological quandles (\cite{Rubinsztein2007}, \cite{EM2016}, \cite{ESZ2019}, \cite{EGL2023}) 
and manifolds and algebraic varieties with quandle structures (\cite{Ishikawa2018}, \cite{Takahashi2016}). 
A regular $s$-manifold can be seen as a manifold with a ``non-degenerate'' quandle structure. 

In the application of quandles, 
modules over quandles and their cohomology groups play 
an important role (\cite{CJKLS2003}, \cite{AG2003}, \cite{Jackson2005}). 
Modules can be considered also in geometric settings. 
In \cite{Takahashi2021}, 
the author defined the notion of a regular quandle module over a regular $s$-manifold $Q$, 
and proved that there is a correspondence between regular quandle modules over $Q$ 
and regular representations of the associated infinitesimal $s$-manifold $T=T_qQ$ 
(see also \cite{BD2009}). 
Thus, a large part of the study of regular modules over $Q$ can be reduced 
to that of representations of $T$. 

The aim of this paper is to relate the representations of Lie-Yamaguti algebras 
and infinitesimal $s$-manifolds 
to representations of Lie algebras. 
Given a Lie-Yamaguti algebra $T$, 
there is a Lie algebra $L(T)=T\oplus \ider(T)$ 
(see Definitions \ref{def_der0_lt}), 
called the standard enveloping Lie algebra of $T$. 
We are especially interested in the case where $L(T)$ is a semisimple Lie algebra, 
and in that case we will say that $T$ is \emph{$L$-semisimple}. 
By a representation of the triple $(L(T), T, \ider(T))$, 
we mean a representation $M$ of $L(T)$ with a decomposition $M=M_{\mathrm{n}}\oplus M_{\mathrm{s}}$ 
satisfying $T\cdot M_{\mathrm{s}}\subseteq M_{\mathrm{n}}$, 
$\ider(T)\cdot M_{\mathrm{n}}\subseteq M_{\mathrm{n}}$ and 
$\ider(T)\cdot M_{\mathrm{s}}\subseteq M_{\mathrm{s}}$. 
Then one can define a representation of $T$ on $M_{\mathrm{n}}$ in a natural way. 
Is it true that a representation of $T$ always is isomorphic 
to a representation constructed in this way? 
The answer is negative in general, 
even if $T$ is $L$-semisimple (Example \ref{ex_non_tight}). 

In this paper, we show that if 
$T$ is $L$-semisimple and $[T, T, T]=T$, then the answer is positive. 
More precisely, we define the category 
$\mathrm{Rep}^{\mathrm{em}}(L(T), T, \ider(T))$ 
of effective and minimal representations of the triple $(L(T), T, \ider(T))$, 
and show that if $T$ is a Lie-Yamaguti algebra such that $L(T)$ is semisimple and $[T, T, T]=T$, 
then $\mathrm{Rep}^{\mathrm{em}}(L(T), T, \ider(T))$ is equivalent 
to the category $\mathrm{Rep}(T)$ of representations of $T$ (Corollary \ref{cor_equivalence_ly}). 

Similarly, if $(T, \sigma)$ is an infinitesimal $s$-manifold, 
we consider a certain category 
$\mathrm{Rep}^{\mathrm{em}}(L(T), L(\sigma))$ 
of representations of $L(T)$ with additional data, 
and show that if $L(T)$ is semisimple, 
then $\mathrm{Rep}^{\mathrm{em}}(L(T), L(\sigma))$ is equivalent 
to the category $\mathrm{Rep}(T, \sigma)$ of representations of $(T, \sigma)$ 
(Corollary \ref{cor_equivalence_ism}). 

We note that there is a certain parallel between our results 
and the classification of irreducible modules over a quandle 
formed by a conjugation class in $SL_2(\mathbb{F}_q)$, 
where $q$ is a power of a prime (\cite{Uematsu2023}). 

This paper is organized as follows. 
In section 2, we recall the definitions of Lie-Yamaguti algebras 
and infinitesimal $s$-manifolds and state a few basic facts. 
In particular, we recall related notions 
such as reductive triples, local regular $s$-pairs, quandles and regular $s$-manifolds 
and explain how they are related. 

In section 3, 
we define the $L$-semisimplicity and $L$-simplicity of Lie-Yamaguti algebras and infinitesimal $s$-manifold, 
and show that a Lie-Yamaguti algebra or an infinitesimal $s$-manifold is 
$L$-semisimple if (and only if) it has a semisimple enveloping Lie algebra, 
not necessarily the standard one. 
A couple of interesting examples are also discussed. 

In section 4, 
we recall the definition of a representation of 
a Lie-Yamaguti algebra (resp. an infinitesimal $s$-manifold), 
define a representation of a reductive triple (resp. a local regular $s$-pair), 
and show how the latter gives rise to the former. 

To construct the inverse correspondence, 
in section 5, we define the notion of tightness of a representation 
of a Lie-Yamaguti algebra 
and give sufficient conditions for tightness, along with a few examples. 

In section 6, 
we prove the correspondence of 
effective minimal representations of $(L(T), T, \ider(T))$ 
and tight representations of $T$. 
In particular, for an $L$-semisimple $T$, we obtain the main results.

\section*{Acknowledgements}
The author would like to thank the referee for valuable comments and suggestions. 
This work was partially supported by JSPS KAKENHI Grant Number \linebreak
JP22K03229.

\section{Lie-Yamaguti algebras and infinitesimal $s$-manifolds}

We work over a field $\mathbb{K}$ 
of characteristic $0$ in this paper.

\subsection{Lie-Yamaguti algebras}

\begin{definition}[{\cite[\S1]{Yamaguti1969}}]
\label{def_lyalg}
(1)
A \emph{Lie-Yamaguti algebra} over $\mathbb{K}$ is a triplet $(T, *, [\ ])$, 
where $T$ is a finite dimensional $\mathbb{K}$-vector space, 
$*: T\times T\to T$ is a bilinear operation, 
and $[\ ]: T\times T\times T\to T$ is a trilinear operation, 
such that the following hold for any $x, y, z, v, w\in T$. 
\begin{enumerate}
\item[(LY1)]
$x*x=0$. 
\item[(LY2)]
$[x, x, y]=0$. 
\item[(LY3)]
$[x, y, z]+[y, z, x]+[z, x, y]+(x*y)*z+(y*z)*x+(z*x)*y=0$. 
\item[(LY4)]
$[x*y, z, w]+[y*z, x, w]+[z*x, y, w]=0$. 
\item[(LY5)]
$[x, y, z*w]=[x, y, z]*w+z*[x, y, w]$. 
\item[(LY6)]
$[x, y, [z, v, w]]=[[x, y, z], v, w]+[z, [x, y, v], w]+[z, v, [x, y, w]]$. 
\end{enumerate}

(2)
A \emph{homomorphism of Lie-Yamaguti algebras} 
from $(T, *, [\ ])$ to $(T', *, [\ ])$ is 
a linear map $f: T\to T'$ 
satisfying $f(x*y)=f(x)*f(y)$ and $f([x, y, z])=[f(x), f(y), f(z)]$. 

(3)
A \emph{derivation} of $T$ is a linear map $u: T\to T$ such that 
$u(x*y)=u(x)*y+x*u(y)$ and $u([x, y, z])=[u(x), y, z]+[x, u(y), z]+[x, y, u(z)]$ 
hold. 

For $x, y\in T$, 
let $D_{x, y}$ denote the map $T\to T$ defined by $D_{x, y}(z)=[x, y, z]$. 
By (LY5) and (LY6), this is a derivation of $T$. 
\end{definition}

Lie-Yamaguti algebras naturally arise from certain decompositions of Lie algebras. 
\begin{definition}
A \emph{reductive triple} is a triplet $(\fg, \fm, \fh)$ 
consisting of a Lie algebra $\fg$, 
a Lie subalgebra $\fh\subseteq\fg$ and a linear subspace $\fm\subseteq\fg$ 
such that $\fg=\fm\oplus \fh$ and $[\fh, \fm]\subseteq \fm$ hold. 
The decomposition $\fg=\fm\oplus \fh$ is called a \emph{reductive decomposition}. 

For a reductive triple $(\fg, \fm, \fh)$ 
and $x\in \fg$, 
we write the $\fm$- and $\fh$-components of $x$ 
by $x_\fm$ and $x_\fh$. 
For a subspace $V$ of $\fg$, 
we write $V_{\fh}$ for the set $\{x_\fh\mid x\in V\}$. 

A \emph{homomorphism of reductive triples} 
from $(\fg, \fm, \fh)$ 
to $(\fg', \fm', \fh')$ 
is a Lie algebra homomorphism $f: \fg\to \fg'$ 
satisfying $f(\fm)\subseteq \fm'$ 
and $f(\fh)\subseteq \fh'$. 
\end{definition}

\begin{proposition}\label{prop_ly_from_rt}
Let $(\fg, \fm, \fh)$ be a reductive triple. 
Then $\fm$ forms a Lie-Yamaguti algebra with the operations 
\begin{eqnarray*}
x*y & := & [x, y]_\fm, \\
 {[}x, y, z{]}  & := & [[x, y]_\fh, z]. 
\end{eqnarray*}
If $f: (\fg, \fm, \fh)\to (\fg', \fm', \fh')$ is a homomorphism of 
reductive triples, 
then the induced map $\fm\to\fm'$ is a homomorphism 
of Lie-Yamaguti algebras. 

In other words, there is a fuctor 
$\mathcal{LY}: \mathrm{RT}\to \mathrm{LY}; (\fg, \fm, \fh)\mapsto \fm$ 
from the category $\mathrm{RT}$ of reductive triples 
to the category $\mathrm{LY}$ of Lie-Yamaguti algebras. 
\end{proposition}
\begin{proof}
The first statement appears in \cite[p. 105--106]{Sagle1968}, 
and the second statement is also well known. 
Both are straightforward to prove. 
\end{proof}

Conversely, any Lie-Yamaguti algebra can be embedded into a Lie algebra. 

\begin{definition}\label{def_der0_lt}
Let $T$ be a Lie-Yamaguti algebra. 

(1)
For linear subspaces $V$ and $W$ of $T$, 
we define the following subspaces of $\mathrm{End}(T)$. 
Here $\langle\dots\rangle$ stands for the linear span. 
\begin{align*}
\ider_T(V, W)
& = \langle \{D_{x, y}\mid x\in V, y\in W\} \rangle, \\
\ider_T(V)
& = \ider_T(V, V)
 = \langle \{D_{x, y}\mid x, y\in V\} \rangle, \\
\ider(T)
& = \ider_T(T)
 = \langle \{D_{x, y}\mid x, y\in T\} \rangle. 
\end{align*}

(2)
The \emph{standard enveloping Lie algebra} $L(T)$ of $T$ is defined 
as the linear space $T\oplus\ider(T)$ with the operation $[, ]$ on $L(T)$ defined by 
\begin{equation}\label{eq_lt}
[(x_1, h_1), (x_2, h_2)] = (x_1 * x_2 + h_1(x_2)-h_2(x_1), D_{x_1, x_2}+[h_1, h_2]). 
\end{equation}
If $(\fg, T, \fh)$ is a reductive triple that induces the given Lie-Yamaguti algebra structure on $T$, 
then $\fg$ is called \emph{an enveloping Lie algebra} of $T$. 
\end{definition}

\begin{remark}
In \cite{Takahashi2021}, $L(T)$ and $\ider(T)$ 
are written as $\fg(T)$ and $\fh(T)$, respectively. 
\end{remark}

\begin{proposition}
(\cite[pp. 61--62]{Nomizu1954}, \cite[Proposition 2.1]{Yamaguti1958})
Let $T$ be a Lie-Yamaguti algebra. 
Then the triple $(L(T), T, \ider(T))$ is a reductive triple, 
and the induced Lie-Yamaguti algebra structure coincides with the original one. 
(Note that the Lie bracket in the above literature differs from ours by a sign.) 
\end{proposition}

\begin{remark}
One subtle point is that a homomorphism $f: T\to T'$ of Lie-Yamaguti algebras 
does not necessarily induce a homomorphism $L(T)\to L(T')$. 
In other words, $T\mapsto L(T)$ does not define 
a functor $\mathrm{LY}\to\mathrm{RT}$ in a natural way. 

This is one of the factors 
that make the representation of Lie-Yamaguti algebras more complicated. 
\end{remark}

Let us look at the relation between a Lie-Yamaguti algebra and its enveloping Lie algebras more closely, 
according to 
\cite{Yamaguti1969}, \cite{Fedenko1977} and \cite{Kowalski1980}. 
We provide some of the proofs 
for the reader's convenience.

\begin{definition}[\cite{Yamaguti1969}]
Let $T$ be a Lie-Yamaguti algebra. 

A \emph{subalgebra} of $T$ is a linear subspace closed under $*$ and $[\ ]$. 

An \emph{ideal} of $T$ is a linear subspace $U$ satisfying 
$U*T\subseteq U$ and $[U, T, T]\subseteq U$. 
It is easy to check that $[T, U, T]\subseteq U$ and $[T, T, U]\subseteq U$ also hold. 

An ideal $U$ of $T$ is \emph{abelian} 
if $U*U=0$ and $[T, U, U]=0$. 
It follows that $[U, T, U]=[U, U, T]=0$.
\end{definition}

\begin{definition}
For a reductive triple $(\fg, \fm, \fh)$ 
and a Lie-Yamaguti subalgebra $\fm'$ of $\fm$, 
define $L_\fg(\fm')=\fm'+[\fm', \fm']$. 
(Here, $[V, W]:=\langle\{[x, y]\mid x\in V, y\in W\}\rangle$.)
\end{definition}

\begin{definition}[{\cite[p.\ 106]{Sagle1968}, \cite[Ch. 1, \S 5, Definition 1]{Fedenko1977}}]
A reductive triple $(\fg, \fm, \fh)$ 
is called \emph{minimal} if $[\fm, \fm]_\fh=\fh$. 
By the following proposition, this is equivalent to $L_{\fg}(\fm)=\fg$. 

The triple is called \emph{effective} if the map 
$\fh\to\mathrm{Hom}(\fm, \fm)$ 
induced by the adjoint action is injective
(see Proposition \ref{prop_isom_to_L} (1) for an equivalent condition). 
\end{definition}

The following facts are also known 
(see \cite[p.\ 106]{Sagle1968}, \cite[\S 1, Proposition]{Yamaguti1969}, 
\cite[Ch. 1, \S 5, Lemma 4]{Fedenko1977}, \cite[Ch. 1, \S 8, Theorem 1]{Fedenko1977}). 

\begin{proposition}\label{prop_L_gives_subalg_ideal}
(1)
Let $(\fg, \fm, \fh)$ 
be a reductive triple. 
If $\fm'$ is a Lie-Yamaguti subalgebra of $\fm$, 
then $L_{\fg}(\fm')$ is a Lie subalgebra of $\fg$. 
The equality 
$L_{\fg}(\fm')=\fm'\oplus[\mathfrak{m'}, \mathfrak{m'}]_{\fh}$ 
holds, 
and $(L_{\fg}(\fm'), \fm', [\fm', \fm']_{\fh})$ 
is a reductive triple. 

(2)
In (1), if furthermore $\fm'$ is an ideal of $\fm$ 
and $(\fg, \fm, \fh)$ is minimal, 
then 
$\fm'+[\fm, \fm']=\fm'\oplus[\fm, \mathfrak{m'}]_{\fh}$ 
is an ideal of $\fg$, 
and $(\fm'+[\fm, \fm'], \fm', [\fm, \fm']_{\fh})$ 
is a reductive triple. 

(3)
$L_\fg(\fm)=\fm+[\fm, \fm]=\fm\oplus [\fm, \fm]_\fh$ is an ideal of $\fg$ 
and $(L_\fg(\fm), \fm, [\fm, \fm]_\fh)$ is a reductive triple. 

(4)
For a Lie-Yamaguti algebra $T$ and its subalgebra $T'$, 
we have $L_{L(T)}(T')=T'\oplus\ider_T(T')$, 
and in particular $L_{L(T)}(T)=L(T)$. 
Thus $L_{L(T)}(T')$ is a Lie subalgebra of $L(T)$ by (1). 
If $T'$ is an ideal, $T'\oplus D_T(T, T')$ is an ideal of $L(T)$ by (2). 
\end{proposition}
\begin{proof}
(1)
Let $\fm'$ be a Lie-Yamaguti subalgebra. 
Then, for $x, y\in \fm'$, 
$[x, y]_{\fm}$ is equal to $x*y$ 
by the definition of the Lie-Yamaguti algebra structure on $\fm$, 
and hence belongs to $\fm'$. 
This implies 
$[x, y]=[x, y]_{\fm}+[x, y]_{\fh}
\in\fm'+[\fm', \fm']_{\fh}$ 
and 
$[x, y]_{\fh}=-[x, y]_{\fm}+[x, y]
\in\fm'+[\fm', \fm']$, 
and 
$\fm'+[\fm', \fm']=\fm'\oplus[\mathfrak{m'}, \mathfrak{m'}]_{\fh}$ 
follows. 

To show that $L_{\fg}(\fm')$ is a Lie subalgebra, 
it suffices to show that (a) $[x, y]\in L_{\fg}(\fm')$ for $x, y\in \fm'$, 
(b) $[[x, y], z]\in L_{\fg}(\mathfrak{m'})$ for $x, y, z\in\fm'$, 
and (c) $[[x, y], [z, w]]\in L_{\fg}(\mathfrak{m'})$ 
for $x, y, z, w\in\fm'$.  

The assertion (a) is clear from $L_{\fg}(\fm')=\fm'+[\fm', \fm']$.
For (b), we have $[[x, y], z]=[[x, y]_{\fm}, z]+[[x, y]_{\fh}, z]$ 
and the first term is in $[\fm', \fm']$ since $[x, y]_{\fm}=x*y\in\fm'$. 
The second term is $[x, y, z]$ by the definition of the triple product on $\fm$, 
which is in $\fm'$. 
From (a) and (b) it follows that 
$[\fm', L_\fg(\mathfrak{m'})]\subseteq L_\fg(\mathfrak{m'})$. 

For (c), 
if we take $x, y, z, w\in\fm'$, by Jacobi identity we have 
\[
[[x, y], [z, w]]= [[[x, y], z], w] + [z, [[x, y], w]], 
\]
and repeatedly applying 
$[\fm', L_\fg(\mathfrak{m'})]\subseteq L_\fg(\mathfrak{m'})$, 
we see that this belongs to $L_{\fg}(\fm')$. 

From the assumption that $(\fg, \fm, \fh)$ 
is a reductive triple 
and equalities $L_\fg(\fm')\cap \fm=\fm'$ 
and $L_{\fg}(\fm')\cap \fh=[\fm', \fm']_{\fh}$, 
it is straightforward to show the statement that 
$(L_{\fg}(\fm'), \fm', [\fm', \fm']_{\fh})$ 
is a reductive triple. 

\smallbreak
(2), (3) 
Let $\fa=\fm'+[\fm, \fm']$, 
with $\fm'=\fm$ in the case of (3). 
As in (1), for $x\in\fm$ and $y\in\fm'$ we have 
$[x, y]_{\fm}\in \fm'$ 
and 
$\fa=\fm'\oplus[\fm, \mathfrak{m'}]_{\fh}$
follows. 
Once we show that $\fa$ is an ideal, 
this gives a reductive decomposition as in (1). 

In (2), 
we see $[\fm, \fm']\subseteq\fa$, 
$[\fm, [\fm, \mathfrak{m'}]]\subseteq\fa$ 
and $[[\fm, \fm], \mathfrak{m'}]\subseteq\fa$ 
as in (1), 
and hence 
$[\fg, \fm']= [\fm, \fm']+[[\fm, \fm], \fm']\subseteq \fa$ 
and $[\fm, \fa]= [\fm, \fm']+[\fm, [\fm, \fm']]\subseteq\fa$. 
Thus, for any $x, y, z\in\fm$ and $w\in\fm'$, 
we have $[[[x, y], z], w]\in[\fg, \fm']\subseteq\fa$ 
and $[z, [[x, y], w]]\in [\fm, [\fg, \fm']]
\subseteq [\fm, \fa]\subseteq\fa$, 
and hence $[[x, y], [z, w]]\in\fa$. 
It follows that 
\[
[\fg, [\fm, \fm']]= [\fm, [\fm, \fm']]+[[\fm, \fm], [\fm, \fm']]\subseteq\fa. 
\]
For (3), from $[\fh, \fm]\subseteq\fm$ it follows that
$[\fg, \fm]= [\fm, \fm]+[\fh, \fm]\subseteq[\fm, \fm]+\fm=\fa$. 
Then we have
\[
[\fg, [\fm, \fm]]\subseteq [[\fg, \fm], \fm]+[\fm, [\fg, \fm]]\subseteq [\fg, \fm]+[\fm, \fg]\subseteq \fa, 
\]
and hence $[\fg, \fa]\subseteq \fa$. 

\smallbreak
(4)
This is obvious 
since the $\ider(T)$-component of $[(x, 0), (y, 0)]$ is $D_{x, y}$ 
by the definition of the Lie bracket on $L(T)$. 
\end{proof}

\begin{proposition}\label{prop_isom_to_L}
(1)
A reductive triple $(\fg, \fm, \fh)$ is effective 
if and only if $\fh$ contains no nonzero ideal of $\fg$. 

(2)
Let $(\fg, \fm, \fh)$ 
be a minimal reductive triple. 
Then any ideal of $\fg$ contained in $\fh$ is 
contained in the center of $\fg$. 
Hence, if $\fg$ is a Lie algebra whose center is trivial 
then the triple is also effective. 

(3)
If $(\fg, \fm, \fh)$ 
is a minimal reductive triple, 
then there exists a unique homomorphism 
$(\fg, \fm, \fh)\to (L(\fm), \fm, \ider(\fm))$ 
restricting to the identity map on $\fm$, 
and it is surjective. 
If the triple is also effective, 
then it is an isomorphism. 

(4)
If $(\fg, \fm, \fh)$ 
is an effective reductive triple, 
then there exists a unique homomorphism 
$(L(\fm), \fm, \ider(\fm))\to(\fg, \fm, \fh)$ 
restricting to the identity map on $\fm$, 
and it is injective. 
If the triple is also minimal, 
then it is an isomorphism. 
\end{proposition}

\begin{proof}
(1)
Let $\fa=\{h\in\fh\mid [h, \fm]=0\}$. 
Then $\fa$ is an ideal of $\fg$: 
obviously $[\fa, \fm]=0$ holds, and 
from
\[
[[\fh, \fa], \fm]\subseteq [\fa, [\fh, \fm]]+[\fh, [\fa, \fm]]
\subseteq [\fa, \fm] + [\fh, 0]=0, 
\]
$[\fh, \fa]\subseteq \fa$ holds. 
Thus, if the triple is not effective, $\fh$ contains a nonzero ideal. 

If $\fa\subseteq\fh$ is an ideal of $\fg$, 
then $[\fa, \fm]\subseteq \fa\cap [\fh, \fm]\subseteq \fa\cap \fm=0$ holds. 
Thus, if $\fa\not=0$, the triple is not effective. 

\smallbreak
(2)
Let $(\fg, \fm, \fh)$ be minimal and $\fa\subseteq \fh$ an ideal of $\fg$. 
Then $[\fa, \fm]=0$ by the proof of (1), 
and by Jacobi identity we also have $[\fa, [\fm, \fm]]=0$. 
Thus $\fa$ is contained in the center of $\fg$. 
In particular, if the center of $\fg$ is trivial, then $\fa=0$ 
and $(\fg, \fm, \fh)$ is effective by (1). 

\smallbreak
(3), (4)
We define a linear map $f: \fg\to \fm\oplus \mathrm{End}(\fm)$ 
by $f(x+h)=(x, \mathrm{ad}(h)|_\fm)$ 
for $x\in\fm$ and $h\in\fh$. 

We first show that $f([x, y]_\fh)=D_{x, y}\in \ider(\fm)$ for $x, y\in\fm$. 
For $z\in \fm$, we have
\[
\mathrm{ad}([x, y]_\fh)(z) = [[x, y]_\fh, z]=[x, y, z] = D_{x, y}(z) 
\]
by the definition of the triple product on $\fm$, 
and $f([x, y]_\fh)=\mathrm{ad}([x, y]_\fh)|_{\fm}=D_{x, y}$ holds. 

If $(\fg, \fm, \fh)$ is minimal, then it follows that 
$\mathrm{Im}(f)=\fm\oplus \ider(\fm)= L(\fm)$, 
and $f$ can be regarded as a surjective map $\fg\to L(\fm)$. 
On the other hand, if $(\fg, \fm, \fh)$ is effective, 
then $f$ is injective, and since $f(L_\fg(\fm))=f(\fm+[\fm, \fm]_\fh)=L(\fm)$, 
there is a right inverse $g: L(\fm)\to \fg$ 
with $\mathrm{Im}(g)=L_\fg(\fm)$. 

Let us show that $f|_{L_\fg(\fm)}$ is a homomorphism of Lie algebras. 
For $x, y\in\fm$, we have 
$f([x, y])=f(x*y+[x, y]_\fh)=(x*y, \mathrm{ad}([x, y]_\fh)|_\fm)$, 
while $[f(x), f(y)]=[(x, 0), (y, 0)]=(x*y, D_{x, y})$, 
and these are equal by what we saw above. 
If $x\in\fm$ and $h\in\fh$, 
$f([x, h])=([x, h], 0)$ since $[x, h]\in\fm$, 
and 
$[f(x), f(h)]=[(x, 0), (0, \mathrm{ad}(h)|_\fm)]=(-\mathrm{ad}(h)|_\fm(x), 0)=(-[h, x], 0)=f([x, h])$. 
If $h, h'\in\fh$, 
$f([h, h'])=(0, \mathrm{ad}([h, h'])|_\fm)$ 
while $[f(h), f(h')]=(0, [\mathrm{ad}(h)|_\fm, \mathrm{ad}(h')|_\fm])$, 
and these are equal by Jacobi identity. 
The calculation for $[f(x), f(y)]$ also shows the uniqueness of $f|_{L_\fg(\fm)}$ or $g$ 
assuming that it is a homomorphism and restricts to the identity on $\fm$. 

In (3), the kernel of $f$ is an ideal included in $\fh$. 
Thus, if $(\fg, \fm, \fh)$ is effective, 
$\mathrm{Ker}\, f$ must be $0$ by (1) and $f$ is an isomorphism onto $L(\fm)$. 
In (4), if $(\fg, \fm, \fh)$ is minimal, 
then $L_\fg(\fm)=\fg$ and $g$ is an isomorphism. 
\end{proof}

\subsection{Infinitesimal $s$-manifolds}

\begin{definition}[{\cite[Definition III.20]{Kowalski1980}}]\label{def_inf_s_mfd}
(1)
An \emph{infinitesimal $s$-manifold} 
is a pair $(T, \sigma)$ of a Lie-Yamaguti algebra $T$ 
and a linear map $\sigma: T\to T$ 
satisfying the following. 
\begin{itemize}
\item[(ISM0)]
Both $\sigma$ and $\mathrm{id}_T-\sigma$ are invertible. 
\item[(ISM1)]
$\sigma(x * y) = \sigma(x) * \sigma(y)$. 
\item[(ISM2)]
$\sigma([x, y, z]) = [\sigma(x), \sigma(y), \sigma(z)]$. 
\item[(ISM3)]
$\sigma([x, y, z]) = [x, y, \sigma(z)]$. 
\end{itemize}

(2)
A \emph{homomorphism of infinitesimal $s$-manifolds} from $(T, \sigma)$ to $(T', \sigma')$ is 
a homomorphism $f: T\to T'$ of Lie-Yamaguti algebras
satisfying $f\circ \sigma=\sigma'\circ f$. 
\end{definition}

Just as a reductive triple gives rise to a Lie-Yamaguti algebra, 
an infinitesimal $s$-manifold is associated to a certain pair of 
a Lie algebra and its automorphism. 
\begin{definition}
(\cite{Fedenko1977}, \cite[Definition III.29]{Kowalski1980})
(1) 
A \emph{local regular $s$-pair} is a pair $(\fg, \varphi)$ 
satisfying the following conditions: 
\begin{itemize}
\item
$\fg$ is a Lie algebra. 
\item
$\varphi$ is an automorphism of $\fg$. 
\item
the fixed part $\fh:=\fg^\varphi$ 
is equal to the generalized eigenspace of $\varphi$ for $1$. 
\end{itemize}

(2)
A \emph{homomorphism $(\fg, \varphi)\to (\fg', \varphi')$ of local regular $s$-pairs} 
is a homomorphism of Lie algebras 
$\fg\to \fg'$ which commutes with $\varphi$ and $\varphi'$. 
\end{definition}

\begin{remark}
In \cite{Fedenko1977}, \cite{Kowalski1980} and \cite{Takahashi2021}, 
the triple $(\fg, \fh, \varphi)$ is considered and is called a local regular $s$-triplet. 
Since it is actually not necessary to specify $\fh$, 
here we consider the pair $(\fg, \varphi)$. 
\end{remark}

\begin{construction}[{See e.g.\ \cite[Construction 2.21]{Takahashi2021}}]\label{const_ism_from_lrst}
Let $(\fg, \varphi)$ be a local regular $s$-pair. 
There is a canonical decomposition $\fg=\fm\oplus\fh$ as a vector space, 
where $\fh=\fg^\varphi$ and $\fm$ is the sum of 
the generalized eigenspaces of $\varphi$ corresponding to eigenvalues different from $1$. 
In other words, $\fh=\mathrm{Ker}(\varphi-1)$ and $\fm=\mathrm{Im}(\varphi-1)$. 
Then $(\fg, \fm, \fh)$ is a reductive triple, 
and we call it the \emph{associated reductive triple}, 
$\fg=\fm\oplus \fh$ the \emph{associated reductive decomposition}. 

We define operations as follows
(see Proposition \ref{prop_ly_from_rt}): 
for $x, y, z\in\fm$, 
\begin{eqnarray*}
\sigma(x) & := & \varphi(x), \\
x*y & := & [x, y]_\fm, \\
 {[}x, y, z{]}  & := & [[x, y]_\fh, z]. 
\end{eqnarray*}
Then $((\fm, *, [\ ]), \sigma)$ is an infinitesimal $s$-manifold. 
\end{construction}

\medbreak

We recall the definitions of quandles, smooth quandles 
and regular $s$-manifolds, 
although they are only needed for examples. 

\begin{definition}\label{def_quandle}
A \emph{quandle} is a set $X$ endowed with a binary operation $\rhd$ 
satisfying the following conditions. 
\begin{enumerate}
\item
For any $x, y\in X$, 
there exists a unique element $y'\in X$ such that $x\rhd y'=y$. 
\item
For any $x, y, z\in X$, 
$x\rhd(y\rhd z)=(x\rhd y)\rhd(x\rhd z)$ holds. 
\item
For any $x\in X$, $x\rhd x=x$ holds. 
\end{enumerate}
We define $s_x: X\to X$ by $s_x(y)=x\rhd y$, 
which is bijective by (1). 
\end{definition}

\begin{definition}\label{def_reg_s_mfd}
(1)
A \emph{smooth quandle} is a smooth manifold $Q$ 
equipped with a quandle operation $\rhd$ 
such that $Q\times Q\to Q\times Q; (q, r)\mapsto (q, q\rhd r)$ 
is a diffeomorphism. 

(2)
(\cite[Definition II.2]{Kowalski1980})
A \emph{regular $s$-manifold} is a smooth quandle $Q$ 
such that $\mathrm{id}_{T_qQ}-d_qs_q\in\mathrm{End}(T_qQ)$ is invertible 
for any $q\in Q$. 
\end{definition}

The following theorem 
describes connected regular $s$-manifolds 
in terms of Lie groups. 
See \cite[Chapter II]{Kowalski1980} for details, 
or \cite[Theorem 2.15]{Takahashi2021} for a summary. 

\begin{theorem}[{cf. \cite[Theorem II.32]{Kowalski1980}}]\label{thm_reg_s_mfd}
If $Q$ is a connected regular $s$-manifold and $q\in Q$ is a point, 
then there is a connected Lie group $G$ acting on $Q$ 
and an automorphism $\Phi$ of $G$ 
with the following properties. 
\begin{itemize}
\item
The action of $G$ on $Q$ is transitive, 
and $\Phi$ acts trivially on $H:=\mathrm{Stab}_G(q)$. 
\item
$Q$ is isomorphic to $G/H$ with the operation $\rhd_{\Phi}$ 
defined by 
\[
xH\rhd_{\Phi} yH:=x\Phi(x^{-1}y)H. 
\]
\end{itemize}
\end{theorem}

We explain how to define a structure of infinitesimal $s$-manifold 
on $T_qQ$ from the above description. 
There is also a direct construction from differential geometry (\cite{Kowalski1980}). 

\begin{construction}\label{const_ism_from_rsm}
For a connected regular $s$-manifold $Q$, 
let $G$ and $\Phi$ be as in the previous theorem. 
Let $\fg:=\mathrm{Lie}(G)$ be the associated Lie algebra 
and $\varphi:=d_e\Phi$. 
Then $(\fg, \varphi)$ is a local regular $s$-pair. 

Applying Construction \ref{const_ism_from_lrst}, 
we have a reductive triple $(\fg, \fm, \fh)$ 
and then a structure of Lie-Yamaguti algebra on $\fm$, 
which together with $\sigma:=\varphi|_\fm$ 
forms an infinitesimal $s$-manifold. 
Note that $\fm$ and $\sigma$ can be identified with $T_qQ$ and $d_qs_q$, 
and $\fh=\mathrm{Lie}(H)$ for the Lie subgroup $H\subseteq G$ in the previous theorem. 
\end{construction}

\section{$L$-Semisimplicity}

As in the case of Lie algebras, 
it is expected that representations of a Lie-Yamaguti algebra $T$ are easier to understand 
when $T$ is ``simple'' or ``semisimple'' in a certain sense. 
In this paper, we consider the condition that $L(T)$ is simple or semisimple. 

\begin{definition}
We say that a Lie-Yamaguti algebra $T$ is \emph{$L$-semisimple} (resp. \emph{$L$-simple})
if its standard enveloping Lie algebra $L(T)$ is semisimple (resp. simple).

We say that an infinitesimal $s$-manifold $(T, \sigma)$ is \emph{$L$-semisimple} (resp. \emph{$L$-simple})
if its underlying Lie-Yamaguti algebra $T$ is $L$-semisimple (resp. \emph{$L$-simple}). 
\end{definition}

We prove that ``standard enveloping Lie algebra'' in the definition can be replaced 
by any enveloping Lie algebra. 

\begin{proposition}\label{prop_semisimple}
(1)
A Lie-Yamaguti algebra $T$ is $L$-semisimple 
if (and only if) there is a reductive triple $(\fg, \fm, \fh)$ 
such that $\fg$ is a semisimple Lie algebra 
and $T\cong \fm$ as Lie-Yamaguti algebras. 

(2) (\cite[Proposition 1.2 (ii)]{BDE2005})
A Lie-Yamaguti algebra $T$ is $L$-simple 
if (and only if) there is a reductive triple $(\fg, \fm, \fh)$ 
such that $\fg$ is simple, $\fm\not=0$ 
and $T\cong \fm$ as Lie-Yamaguti algebras. 
\end{proposition}

This is immediate from the following more precise statement. 

\begin{lemma}\label{lem_semisimple}
Let $\fg$ be a semisimple Lie algebra and 
$(\fg, \fm, \fh)$ a reductive triple. 

(1)
There is an isomorphism 
$(L(\fm), \fm, \ider(\fm))\overset{\sim}{\to}(L_\fg(\fm), \fm, [\fm, \fm]_\fh)$ 
of reductive triples 
that extends the identity map on $\fm$, 
and $L_\fg(\fm)$ is an ideal of $\fg$. 
Hence $L(\fm)$ is semisimple. 

(2) (\cite[Proposition 1.2 (ii)]{BDE2005})
If $\fg$ is simple and $\fm\not=0$, 
then $L_\fg(\fm)=\fg$ 
and hence $(L(\fm), \fm, \ider(\fm))\cong (\fg, \fm, \fh)$ holds. 
\end{lemma}

\begin{proof}
From Proposition \ref{prop_L_gives_subalg_ideal} (3), 
$L_\fg(\fm)=\fm\oplus[\fm, \fm]_\fh$ 
is an ideal of $\fg$ and 
$(L_\fg(\fm), \fm, [\fm, \fm]_\fh)$ is a minimal reductive triple. 
If $\fg$ is semisimple (resp. $\fg$ is simple and $\fm\not=0$), 
its ideal $L_\fg(\fm)$ is also semisimple 
(resp. $L_\fg(\fm)=\fg$, 
since $L_\fg(\fm)\supseteq\fm\not=0$). 
In particular, the center of $L_\fg(\fm)$ is trivial, 
and $L_\fg(\fm)$ is isomorphic to $L(\fm)$ by Proposition \ref{prop_isom_to_L} (2), (3). 
\end{proof}

\begin{example}[Conjugation quandles]
Let $G$ be a connected Lie group 
and $A\in G$ an element. 
Write $\fg$ for the Lie algebra of $G$, 
and $\varphi:=\mathrm{Ad}_\fg(A)$. 
Assume that the fixed part $\fh=\fg^\varphi$ 
is equal to the generalized eigenspace of $\varphi$ corresponding to $1$. 

Then the conjugacy class $Q_A$ of $A$ is a connected smooth $s$-manifold 
with the operation $X\rhd Y=XYX^{-1}$. 
The group $G$ acts on $Q_A$ by conjugation $M\cdot X=MXM^{-1}$, 
and the stabilizer of $A$ is the centralizer $H=\{M\in G\mid MA=AM\}$. 
If $\Phi: G\to G$ denotes the inner automorphism $M\mapsto AMA^{-1}$, 
then $Q_A$ is isomorphic to the regular $s$-manifold associated to $(G, H, \Phi)$, 
i.e. $G/H$ endowed with the operation $MH\rhd NH=M\Phi(M^{-1}N)H$. 
Thus the associated infinitesimal $s$-manifold 
is given by the local regular $s$-pair $(\fg, \varphi)$, 
with the associated reductive decomposition $\fg=\fm\oplus \fh$ 
where $\fm:=\mathrm{Im} (\varphi-1)$. 

Now assume that $\fg$ is semisimple. 
The triple $(\fg, \fm, \fh)$ is not necessarily minimal (take $A=I$ for example), 
so $L(\fm)$ may not be isomorphic to $\fg$, 
but it is isomorphic to an ideal of $\fg$ and therefore 
$\fm$ is always $L$-semisimple. 

If $\fg$ is simple, 
then we have $\fg\cong L(\fm)$ unless $\varphi=1$, i.e., $A\in Z(G)$. 
This applies 
to $G=\mathrm{SL}_n(\mathbb{R})$ and a non-scalar diagonalizable matrix $A$, 
for example.
\end{example}

\begin{example}[Lie-Yamaguti algebra structures on a Lie algebra]
\label{ex_l_to_ly}

Let $\fg$ be a Lie algebra. 
For $\alpha, \beta\in\mathbb{K}$ and $x, y, z\in\fg$, 
define $x*_\alpha y:=\alpha[x, y]$ and $[x, y, z]_\beta:=\beta[[x, y], z]$. 
Then it is easy to check that 
$\fg_{\alpha, \beta}:=(\fg, *_\alpha, [\ ]_\beta)$ is a Lie-Yamaguti algebra. 
Since $\fg_{\alpha, \beta}\cong \fg_{u\alpha, u^2\beta}$ for $u\not=0$ by scaling, 
each such Lie-Yamaguti algebra is isomorphic to $\fg_{0, 0}$, 
$\fg_{1, \beta}$ or $\fg_{0, 1}$. 

Assume that $\fg$ is a nonzero semisimple Lie algebra. 
We will show that $\fg_{\alpha, \beta}$ is $L$-semisimple if and only if $\alpha^2+4\beta\not=0$ 
(cf. \cite[Theorem 2.4]{BEMH2009}). 
Since $L(\fg_{0, 0})$ is an abelian Lie algebra, 
it suffices to show that 
$L(\fg_{1, \beta})$ is semisimple if and only if $\beta\not=-1/4$ 
and that $L(\fg_{0, 1})$ is semisimple. 

Let us write $\fg_\beta:=\fg_{1, \beta}$. 
If $\beta=0$, then the triple product is trivial 
and $L(\fg_0)$ is $\fg$, with the original Lie bracket, which is semisimple. 

Assume $\beta\not=0$ with $\alpha$ arbitrary. 
By the definition of the triple product, 
we see that 
$\fg\otimes\fg\to \ider(\fg_{\alpha, \beta}); x\otimes y\mapsto D_{x, y}$ 
factors into $\fg\otimes\fg\to [\fg, \fg]; x\otimes y\mapsto \beta[x, y]$ 
and $[\fg, \fg]\to \ider(\fg_{\alpha, \beta}); x\mapsto \mathrm{ad}_x$. 
Since $\fg$ is semisimple, 
we have $[\fg, \fg]=\fg$ 
and the latter map gives an identification of 
$\ider(\fg_{\alpha, \beta})$ with $\fg$ 
as Lie algebras. 
Thus we have a bijective linear map 
$\varphi: L(\fg_{\alpha, \beta})\to \fg\oplus\fg$ 
characterized by $\varphi((x, D_{y, z}))=(x, \beta [y, z])$, 
and the Lie bracket on $L(\fg_{\alpha, \beta})$ induces 
the operation 
\[
[(x_1, x_2), (y_1, y_2)] = (\alpha[x_1, y_1] + [x_1, y_2] + [x_2, y_1], \beta[x_1, y_1]+[x_2, y_2]) 
\]
on $\fg\oplus\fg$ via $\varphi$. 
If $\alpha=1$ and $\beta\not=0, -1/4$, let $\mu, \nu$ be the roots of $X^2-X-\beta=0$ 
and define a bijective map $\psi: \fg\oplus \fg\to \fg\oplus \fg$ by $(x_1, x_2)\mapsto (\mu x_1+x_2, \nu x_1+x_2)$. 
If we regard the right hand side as the direct sum of Lie algebras, 
then it is straightforward to see that 
$\psi\circ\varphi: L(\fg_\beta)\to \fg\oplus \fg$ 
is an isomorphism of Lie algebras mapping 
$\mathfrak{\fg_\beta}$ and $\ider(\fg_\beta)$ 
to $\{(\mu x, \nu x)\mid x\in\fg\}$ 
and $\{(x, x)\mid x\in\fg\}$, respectively. 

In the case $\alpha=1$ and $\beta=-1/4$, we define 
\[
\psi: \fg\oplus \fg \to \fg_{0, 0}\rtimes_{\mathrm{ad}} \fg; 
\ (x_1, x_2)\mapsto \left(x_1, \frac{1}{2}x_1+x_2\right), 
\]
where $\fg_{0, 0}$ is regarded as a Lie algebra with the trivial bracket 
and the right hand side has the Lie bracket 
$[(x_1, x_2), (y_1, y_2)]=([x_2, y_1] + [x_1, y_2], [x_2, y_2])$. 
Then 
$\psi\circ\varphi: L(\fg_{-1/4})\to \fg_{0, 0}\rtimes_{\mathrm{ad}} \fg$ 
is an isomorphism of Lie algebras. 
Thus $L(\fg_{-1/4})$ is not semisimple. 
In this case, 
$\fg_{-1/4}$ and $\ider(\fg_{-1/4})$ 
are mapped to $\{(x, x/2)\mid x\in\fg\}$ 
and $\{(0, x)\mid x\in\fg\}$. 

In the case $\alpha=0$ and $\beta=1$, we define a bijective map 
$\psi: \fg\oplus \fg\to \fg\oplus \fg$ by $(x_1, x_2)\mapsto (x_1+x_2, -x_1+x_2)$. 
Then 
$\psi\circ\varphi: L(\fg_{0, 1})\to \fg\oplus \fg$ 
is an isomorphism 
mapping $\fg_{0, 1}$ and $\ider(\fg_{0, 1})$ 
to $\{(x, -x)\mid x\in\fg\}$ 
and $\{(x, x)\mid x\in\fg\}$, 
and $L(\fg_{0, 1})$ is semisimple. 
This case can be regarded as the limit of $L(\fg_\beta)$ as $\beta\to\infty$. 

If $G$ is a connected Lie group whose Lie algebra is $\fg$, 
the last case corresponds to the core quandle $(G, \rhd)$ 
defined by $X\rhd Y=XY^{-1}X$. 
Indeed, this quandle is isomorphic to the homogeneous quandle 
$((G\times G)/\Delta, \rhd_\Phi)$, 
where $\Delta$ is the diagonal, 
$\Phi: G\times G\to G\times G;\ (g_1, g_2)\mapsto (g_2, g_1)$ 
and $\rhd_\Phi$ is given by $g\Delta \rhd_\Phi g'\Delta=g\Phi(g^{-1}g')\Delta$. 
The tangent space to $G\times G$ is 
$\fg\oplus\fg$, 
the tangent map $\varphi=d\Phi_{(e, e)}$ is $(x_1, x_2)\mapsto (x_2, x_1)$, 
the eigenspace of $\varphi$ corresponding to $1$ is $\{(x, x)\mid x\in\fg\}$ and 
the eigenspace corresponding to $-1$ is 
$\{(x, -x)\mid x\in\fg\}$. 
\end{example}

To conclude this section, 
let us give a number of remarks on notions related to $L$-simplicity and $L$-semisimplicity. 

\begin{remark}
(1)
Let us compare simplicity, irreducibility (see \cite{BEMH2009}) and $L$-simplicity 
of a Lie-Yamaguti algebra $T$. 

(i) (\cite[Proposition 1.2 (i)]{BDE2005})
If $T$ is $L$-simple, then $T$ is simple in the sense that 
$T$ has no ideal other than $0$ and $T$. 
This is immediate from Proposition \ref{prop_L_gives_subalg_ideal} (4). 

(ii)
$L$-simplicity does not imply irreducibility. 
Let $\fg$ be a simple Lie algebra 
and $T=\fg_{1, 0}$ with the notation of the previous example, 
i.\,e., $\fg$ regarded as a Lie-Yamaguti algebra with the trivial triple product. 
Then $T$ is $L$-simple. 
On the other hand, $\ider(T)=0$ 
and any subspace of $T$ is a $\ider(T)$-subspace. 

There are also examples with nontrivial triple products. 
Let $\fg=\mathfrak{sl}_2$ with the standard basis $e, f, h$, 
$\fm=\langle e, f\rangle$ and $\fh=\langle h\rangle$. 
Then $(\fg, \fm, \fh)$ 
is a reductive triple. 
For the induced Lie-Yamaguti algebra structure on $\fm$, 
$(L(\fm), \fm, \ider(\fm))\cong (\fg, \fm, \fh)$ holds.  
Thus $\fm$ is $L$-simple. 
On the other hand, $\langle e\rangle\subset\fm$ gives a nontrivial $\ider(\fm)$-subspace. 

(iii)
If $T$ is irreducible, then $T$ is simple. 
Assume that $T$ is irreducible. 
If $U$ is an ideal of $T$, 
then $[T, T, U]\subseteq U$ holds, that is, 
$U$ is closed under the action of $\ider(T)$. 
From the definition of irreducibility, $U$ must be $0$ or $T$. 

(iv)
Irreducibility does not imply $L$-simplicity. 
Let $\fg$ be a simple Lie algebra again, 
and let $T=\fg_{1, \beta}$ with $\beta\not=0, -1/4$. 
Then $T$ is an irreducible Lie-Yamaguti algebra of adjoint type 
(see \cite[Theorem 2.4]{BEMH2009}) 
and $L(T)\cong \fg\oplus\fg$ is not simple. 

(v)
Simplicity does not imply $L$-simplicity by (iii) and (iv). 

(vi)
Simplicity does not imply irreducibility by (i) and (ii). 

\smallbreak
(2)
For Lie triple systems, 
Lister (\cite{Lister1952}) 
defines the notion of solvability of an ideal, 
and then that of semisimplicity of a Lie triple system. 
Kikkawa generalized this to Lie-Yamaguti algebras, 
and proved that if a Lie-Yamaguti algebra is $L$-semisimple, 
then it is semisimple in the sense that 
its radical is zero (\cite[Theorem 3]{Kikkawa1979}). 
\end{remark}

\section{Representations}

We recall the definition of representations of 
Lie-Yamaguti algebras and infinitesimal $s$-manifolds. 
Then we define representations of reductive triples and local regular $s$-pairs, 
and explain how they give rise to representations of Lie-Yamaguti algebras 
and infinitesimal $s$-manifolds, respectively.

The following notion of a representation 
of a Lie-Yamaguti algebra was defined 
in \cite[\S6]{Yamaguti1969}. 

\begin{definition}\label{def_rly}
A \emph{representation of a Lie-Yamaguti algebra} $T$ is 
a quadruplet $(V, \rho, \delta, \theta)$, 
where $V$ is a finite dimensional vector space, $\rho: T\to\mathrm{End}(V)$ 
is a linear map and $\delta, \theta: T\times T\to\mathrm{End}(V)$ 
are bilinear maps, 
such that the following hold for any $x, y, z, w\in T$. 
\begin{enumerate}
\item[(RLY1)]
$\delta(x, y)+\theta(x, y)-\theta(y, x)=[\rho(x), \rho(y)]-\rho(x*y)$. 
\item[(RLY2)]
$\theta(x, y*z)-\rho(y)\theta(x, z)+\rho(z)\theta(x, y)=0$. 
\item[(RLY3)]
$\theta(x*y, z)-\theta(x, z)\rho(y)+\theta(y, z)\rho(x)=0$. 
\item[(RLY4)]
$\theta(z, w)\theta(x, y)-\theta(y, w)\theta(x, z)-\theta(x, [y, z, w])
+ \delta(y, z)\theta(x, w)=0$. 
\item[(RLY5)]
$[\delta(x, y), \rho(z)]=\rho([x, y, z])$. 
\item[(RLY6)]
$[\delta(x, y), \theta(z, w)] = \theta([x, y, z], w)+\theta(z, [x, y, w])$. 
\end{enumerate}
A \emph{homomorphism} of representations of Lie-Yamaguti algebras 
is a linear map compatible with $\rho, \delta$ and $\theta$. 
\end{definition}

\begin{remark}
(1)
As noted in \cite{Yamaguti1969}, 
we may just think about $(V, \rho, \theta)$ and define $\delta$ by (RLY1). 

(2)
We need finite dimensionality in the proof of Corollary \ref{cor_ss_section}, 
where we apply Levi's theorem. 
\end{remark}

\begin{definition}[\cite{Takahashi2021}]
\label{def_rism}
A \emph{representation of an infinitesimal $s$-manifold} $(T, \sigma)$ 
is a data $(V, \rho, \delta, \theta, \psi)$ 
where $(V, \rho, \delta, \theta)$ is a representation  of $T$ 
and $\psi\in \mathrm{End}(V)$ is an invertible linear transformation 
satisfying the following for any $x, y\in T$: 
\begin{enumerate}
\item[(RISM1)]
$\rho(\sigma(x))=\psi\circ\rho(x)\circ\psi^{-1}$. 
\item[(RISM2)]
$\theta(x, \sigma(y)) = \psi\circ \theta(x, y)$, \ 
$\theta(\sigma(x), y) = \theta(x, y)\circ\psi^{-1}$. 
\item[(RISM3)]
$\delta(x, y) = \psi\circ \delta(x, y)\circ\psi^{-1}$. 
\end{enumerate}
We will also refer to the representation as $(V, \psi)$. 
A representation $(V, \psi)$ is called \emph{regular} if $V$ is finite dimensional and 
$\mathrm{id}_V-\psi$ is invertible. 

A \emph{homomorphism} of representations of an infinitesimal $s$-manifold 
is defined as a homomorphism of representations of Lie-Yamaguti algebras 
commuting with $\psi$. 
\end{definition}

Recall that for a reductive triple $(\fg, \fm, \fh)$ 
there is a natural structure of Lie-Yamaguti algebra on $\fm$. 
In order to study representations of $\fm$ in terms of $(\fg, \fm, \fh)$, 
let us start by defining representations of the triple $(\fg, \fm, \fh)$.

\begin{definition}\label{def_rep_rt}
Let $(\fg, \fm, \fh)$ 
be a reductive triple. 

(1)
A \emph{representation} of the reductive triple $(\fg, \fm, \fh)$ 
is defined as a pair of a finite dimensional representation $M$ of $\fg$ 
and a decomposition $M=M_{\mathrm{n}}\oplus M_{\mathrm{s}}$ 
such that $\fh M_{\mathrm{n}}\subseteq M_{\mathrm{n}}$, $\fh M_{\mathrm{s}}\subseteq M_{\mathrm{s}}$ 
and $\fm M_{\mathrm{s}}\subseteq M_{\mathrm{n}}$ hold. 
(Here, ``$\mathrm{s}$'' is meant to stand for ``stabilizing.'')
We also refer to this representation as the representation $(M, M_{\mathrm{n}}, M_{\mathrm{s}})$. 

If $N\subseteq M$ is a subspace, 
not necessarily compatible with the decomposition of $M$, 
we write $N_{\mathrm{n}}$  and $N_{\mathrm{s}}$ for the images of $N$ 
under the projections to $M_{\mathrm{n}}$ and $M_{\mathrm{s}}$. 

We say that $(M, M_{\mathrm{n}}, M_{\mathrm{s}})$ is \emph{effective} if 
the natural map $M_{\mathrm{s}}\to \mathrm{Hom}(\fm, M_{\mathrm{n}})$ is injective, 
and \emph{minimal} if $(\fm M_{\mathrm{n}})_{\mathrm{s}}=M_{\mathrm{s}}$. 

(2)
A \emph{homomorphism} $(M, M_{\mathrm{n}}, M_{\mathrm{s}})\to (M', M'_{\mathrm{n}}, M'_{\mathrm{s}})$ of representations 
is a homomorphism $f: M\to M'$ of representations of $\fg$ 
satisfying $f(M_{\mathrm{n}})\subseteq M'_{\mathrm{n}}$ and $f(M_{\mathrm{s}}) \subseteq M'_{\mathrm{s}}$. 

A \emph{subrepresentation} of $(M, M_{\mathrm{n}}, M_{\mathrm{s}})$ 
is a $\fg$-subrepresentation $N$ of $M$ that is compatible with the decomposition of $M$, 
in the sense that $N=(N\cap M_{\mathrm{n}})+(N\cap M_{\mathrm{s}})$ holds. 
A \emph{quotient representation} is the quotient by a subrepresentation. 
It is easy to see that subrepresentations and quotient representations 
actually have natural structures of representations of $(\fg, \fm, \fh)$. 

Let us write the category of representations of 
a reductive triple $(\fg, \fm, \fh)$ 
as $\mathrm{Rep}(\fg, \fm, \fh)$ 
and its full subcategories of effective, minimal, and effective and minimal representations 
as $\mathrm{Rep}^{\mathrm{e}}(\fg, \fm, \fh)$, 
$\mathrm{Rep}^{\mathrm{m}}(\fg, \fm, \fh)$ and 
$\mathrm{Rep}^{\mathrm{em}}(\fg, \fm, \fh)$. 
\end{definition}

The following is straightforward. 
\begin{proposition}
The category $\mathrm{Rep}(\fg, \fm, \fh)$ is abelian. 
The full subcategories 
$\mathrm{Rep}^{\mathrm{e}}(\fg, \fm, \fh)$, 
$\mathrm{Rep}^{\mathrm{m}}(\fg, \fm, \fh)$ and 
$\mathrm{Rep}^{\mathrm{em}}(\fg, \fm, \fh)$ are 
closed under the operations of taking direct sums and direct summands. 
\end{proposition}

\begin{lemma}
Let $(M, M_{\mathrm{n}}, M_{\mathrm{s}})$ be a representation of 
a reductive triple $(\fg, \fm, \fh)$. 
If $(M, M_{\mathrm{n}}, M_{\mathrm{s}})$ is effective, 
then its subrepresentations are effective. 
If $(M, M_{\mathrm{n}}, M_{\mathrm{s}})$ is minimal, 
then its quotient representations are minimal. 
\end{lemma}
\begin{proof}
If $(N, N_{\mathrm{n}}, N_{\mathrm{s}})$ is a subrepresentation, 
there are inclusion maps $N_{\mathrm{s}}\hookrightarrow M_{\mathrm{s}}$ 
and 
$\mathrm{Hom}(\fm, N_{\mathrm{n}})\hookrightarrow\mathrm{Hom}(\fm, M_{\mathrm{n}})$, 
which commute with the natural maps $N_{\mathrm{s}}\to \mathrm{Hom}(\fm, N_{\mathrm{n}})$ 
and $M_{\mathrm{s}}\to \mathrm{Hom}(\fm, M_{\mathrm{n}})$. 
Thus, if the latter is injective, so is the former. 

If $(N, N_{\mathrm{n}}, N_{\mathrm{s}})$ is a quotient representation 
with the projection $\pi: M\to N$, 
then 
\[
\pi((\fm M_{\mathrm{n}})_{\mathrm{s}})=(\pi(\fm M_{\mathrm{n}}))_{\mathrm{s}}=
(\fm\,\pi(M_{\mathrm{n}}))_{\mathrm{s}}
=(\fm N_{\mathrm{n}})_{\mathrm{s}}, 
\]
and $(\fm M_{\mathrm{n}})_{\mathrm{s}}=M_{\mathrm{s}}$ implies 
$(\fm N_{\mathrm{n}})_{\mathrm{s}}=\pi(M_{\mathrm{s}})=N_{\mathrm{s}}$. 
\end{proof}

\begin{example}
The kernel and cokernel of a homomorphism of effective minimal representations 
are not necessarily effective and minimal, 
and thus $\mathrm{Rep}^{\mathrm{em}}(\fg, \fm, \fh)$ 
is not necessarily an abelian subcategory of 
$\mathrm{Rep}(\fg, \fm, \fh)$. 
(Here, an abelian subcategory of an abelian category 
means a full subcategory which is abelian and 
such that the inclusion functor is exact.) 
Let us give an example where 
$(\fg, \fm, \fh)$ is also effective and minimal. 

Let $\fg=\fm=\mathbb{K}$ 
be an abelian Lie algebra and consider the reductive triple $(\fg, \fm, 0)$. 
Let $M_{\mathrm{n}}=\mathbb{K}^m$, $M_{\mathrm{s}}=\mathbb{K}^n$ and $M=M_{\mathrm{n}}\oplus M_{\mathrm{s}}$. 
A structure of a representation of $(\fg, \fm, 0)$ 
on $(M, M_{\mathrm{n}}, M_{\mathrm{s}})$ 
is given by making $1\in\mathbb{K}=\fg$ act on $M$ 
as 
$\begin{pmatrix}
A & B \\ C & O 
\end{pmatrix}$, 
where $A$, $B$ and $C$ are $m\times m$, $m\times n$ and $n\times m$ matrices. 
It is effective (resp. minimal) if and only if the rank of $B$ (resp. $C$) is $n$. 

Now let $m=2$, $n=1$, $A=O$, $B=\begin{pmatrix} 1 \\ 0 \end{pmatrix}$ 
and $C=\begin{pmatrix} 0 & 1 \end{pmatrix}$, 
so that
\[
\begin{pmatrix}
A & B \\ C & O 
\end{pmatrix}
= 
\begin{pmatrix}
0 & 0 & 1 \\ 0 & 0 & 0 \\ 0 & 1 & 0 
\end{pmatrix}. 
\]
Then $M$ is effective and minimal. 
The subspaces $M_1=(\mathbb{K}\oplus 0)\oplus 0$ 
and $M_2=(\mathbb{K}\oplus 0)\oplus\mathbb{K}$ 
are invariant under the multiplication by this matrix, 
i.e., they are subrepresentations of $M$. 
We see that $M_1$ is effective and minimal, since its $\mathrm{s}$-part is $0$, 
while $M/M_1\cong \mathbb{K}\oplus\mathbb{K}$ (as a vector space) 
is not effective, since the action of $1\in\fm$ is given by 
$\begin{pmatrix}
0 & 0 \\ 1 & 0 
\end{pmatrix}$. 
Similarly, 
$M/M_2\cong \mathbb{K}\oplus 0$ is effective and minimal 
while $M_2\cong \mathbb{K}\oplus\mathbb{K}$ is not minimal since it is given by the matrix 
$\begin{pmatrix}
0 & 1 \\ 0 & 0 
\end{pmatrix}$. 
\end{example}

An arbitrary representation is related to effective and minimal representations 
in the following way. 
\begin{lemma}
Let $(M, M_{\mathrm{n}}, M_{\mathrm{s}})$ be a representation of a reductive triple $(\fg, \fm, \fh)$. 

(1)
There is an effective representation $(M', M_{\mathrm{n}}, M'_{\mathrm{s}})$, 
an effective and minimal representation $(M'', M_{\mathrm{n}}, M''_{\mathrm{s}})$ 
and a surjective homomorphism $M\to M'$ and an injective homomorphism $M''\to M'$ 
which are identity maps on $M_{\mathrm{n}}$. 

(2)
There is a minimal representation $(M', M_{\mathrm{n}}, M'_{\mathrm{s}})$, 
an effective and minimal representation $(M'', M_{\mathrm{n}}, M''_{\mathrm{s}})$ 
and an injective homomorphism $M'\to M$ and a surjective homomorphism $M'\to M''$ 
which are identity maps on $M_{\mathrm{n}}$. 
\end{lemma}
\begin{proof}
We consider the following constructions. 

(a) 
Let $K:=\mathrm{Ker}(M_{\mathrm{s}}\to\mathrm{Hom}(\fm, M_{\mathrm{n}}))$. 
Then $\fm K=0$ and 
\[
\fm(\fh K)\subseteq
[\fm, \fh]K+\fh(\fm K)
\subseteq \fm K=0, 
\]
and hence $\fh K\subseteq K$. 
Thus $K$ is a subrepresentation of $M$, and 
if we take $M'_{\mathrm{s}}:=M_{\mathrm{s}}/K$ and $M'=M_{\mathrm{n}}\oplus M'_{\mathrm{s}}$, 
it is easy to see that $(M', M_{\mathrm{n}}, M'_{\mathrm{s}})$ is an effective representation. 
It is effective and minimal if $(M, M_{\mathrm{n}}, M_{\mathrm{s}})$ is minimal. 

(b)
Let $M'_{\mathrm{s}}:=(\fm M_{\mathrm{n}})_{\mathrm{s}}$ 
and $M':=M_{\mathrm{n}}\oplus M'_{\mathrm{s}}$. 
We have 
\[
\fh(\fm M_{\mathrm{n}})
\subseteq [\fh, \fm]M_{\mathrm{n}} + \fm(\fh M_{\mathrm{n}})
\subseteq\fm M_{\mathrm{n}}, 
\]
and therefore 
\[
\fh M'_{\mathrm{s}}=
\fh(\fm M_{\mathrm{n}})_{\mathrm{s}}
= (\fh(\fm M_{\mathrm{n}}))_{\mathrm{s}}
\subseteq(\fm M_{\mathrm{n}})_{\mathrm{s}}
=M'_{\mathrm{s}}, 
\]
while $\fm M'_{\mathrm{s}}\subseteq \fm M_{\mathrm{s}}\subseteq M_{\mathrm{n}}$. 
Thus $(M', M_{\mathrm{n}}, M'_{\mathrm{s}})$ is a subrepresentation, 
which is obviously minimal. 
It is effective and minimal if $(M, M_{\mathrm{n}}, M_{\mathrm{s}})$ is effective. 

By taking the construction (a) and then (b), we have (1). 
By taking (b) and then (a), we have (2). 
\end{proof}

Given a representation $(M, M_{\mathrm{n}}, M_{\mathrm{s}})$ 
of $(\fg, \fm, \fh)$, 
we can think of $M_{\mathrm{n}}$ as a representation of $\fm$, 
regarded as a Lie-Yamaguti algebra, in a natural way. 
To see this, let us recall the following correspondence of 
representations of a Lie-Yamaguti algebra $T$ 
and certain extensions of $T$. 
\begin{proposition}[{\cite[Proposition 4.5]{Takahashi2021}}]\label{prop_ly_rep_ext}
Let $T$ be a Lie-Yamaguti algebra. 

(1) 
Let $(V, \rho, \theta, \delta)$ be a representation of $T$. 
Then $T\oplus V$ is a Lie-Yamaguti algebra by the operations 
\begin{equation}\label{eq_rep_cor}
\begin{array}{rcl}
(x_1, v_1)*(x_2, v_2) & = & (x_1*x_2, \rho(x_1)v_2-\rho(x_2)v_1), \\
{[}(x_1, v_1), (x_2, v_2), (x_3, v_3)] & = & 
([x_1, x_2, x_3], \\
& & \ \theta(x_2, x_3)v_1-\theta(x_1, x_3)v_2+\delta(x_1, x_2)v_3), 
\end{array}
\end{equation}
and $T\oplus 0$ is a subalgebra and $0\oplus V$ is an abelian ideal. 

(2) 
Conversely, if $T\oplus V$ is given a structure of a Lie-Yamaguti algebra 
such that $T\oplus 0$ is a subalgebra with the same operations as $T$ 
and $0\oplus V$ is an abelian ideal, 
then the Lie-Yamaguti algebra structure is given in the form of (\ref{eq_rep_cor})
and $(V, \rho, \theta, \delta)$ is a representation of $T$. 
\end{proposition}

\begin{proposition}\label{prop_rrt_to_rly}
Let $(\fg, \fm, \fh)$ be a reductive triple. 

(1)
Given a representation $(M, M_{\mathrm{n}}, M_{\mathrm{s}})$ of $(\fg, \fm, \fh)$, 
the triple 
\[
(\tilde{\fg}, \tilde{\fm}, \tilde{\fh})=
(\fg\oplus M, \fm\oplus M_{\mathrm{n}}, \fh\oplus M_{\mathrm{s}})
\]
is a reductive triple. 
The associated Lie-Yamaguti algebra structure on $\fm\oplus M_{\mathrm{n}}$ is given by 
\begin{equation}\label{eq_ly_on_ext}
\begin{array}{rcl}
(x_1, m_1)*(x_2, m_2) & = & (x_1*x_2, (x_1m_2-x_2m_1)_{\mathrm{n}}), \\
{[}(x_1, m_1), (x_2, m_2), (x_3, m_3)] & = & ([x_1, x_2, x_3], \\
 & & \ [x_1, x_2]_\fh m_3 -x_3(x_1m_2-x_2m_1)_{\mathrm{s}}), 
\end{array}
\end{equation}
where we write $m_{\mathrm{n}}$ and $m_{\mathrm{s}}$ 
for the $M_{\mathrm{n}}$- and $M_{\mathrm{s}}$-components of $m\in M$. 

(2)
The projection $\tilde{\fg}\to \fg$ 
and the zero section $\fg\to \tilde{\fg}$ 
are homomorphisms of reductive triples, 
and hence the induced maps $\tilde{\fm}\to \fm$ 
and $\fm\to \tilde{\fm}$
are homomorphisms of Lie-Yamaguti algebras. 

(3)
The kernel $0\oplus M$ of $\tilde{\fg}\to \fg$ 
is an abelian ideal of $\tilde{\fg}$, 
and the kernel $0\oplus M_{\mathrm{n}}$ of $\tilde{\fm}\to \fm$ 
is an abelian ideal of $\tilde{\fm}$. 

Hence, by the correspondence of the previous proposition, 
$M_{\mathrm{n}}$ can be considered as a representation of $\fm$. 
Concretely, 
\begin{equation}\label{eq_rrt_to_rly}
\begin{array}{rcl}
\rho(x)(m) & = & (xm)_{\mathrm{n}}, \\
\theta(x, y)(m) & = & y\cdot (xm)_{\mathrm{s}}, \\
\delta(x, y)(m) & = & [x, y]_{\fh}m. 
\end{array}
\end{equation}
\end{proposition}
\begin{proof}
(1) 
As is well known, 
$\tilde{\fg}=\fg\oplus M$ is a Lie algebra with the bracket 
defined by $[(x_1, m_1), (x_2, m_2)]=([x_1, x_2], x_1m_2-x_2m_1)$. 
The conditions $[\tilde{\fh}, \tilde{\fh}]\subseteq \tilde{\fh}$ 
and $[\tilde{\fh}, \tilde{\fm}]\subseteq \tilde{\fm}$ 
are straightforward from the definition of a representation of a reductive triple. 

Taking the $\fm$-component of $[x_1, x_2]$
and the $M_{\mathrm{n}}$-component of $x_1m_2-x_2m_1$, 
we have the first equality of (\ref{eq_ly_on_ext}). 
The $\tilde{\fh}$-component of $[(x_1, m_1), (x_2, m_2)]$ is 
$([x_1, x_2]_\fh, (x_1m_2-x_2m_1)_{\mathrm{s}})$, 
and we have the second equality of (\ref{eq_ly_on_ext}) 
by taking the Lie bracket with $(x_3, m_3)$. 

\smallbreak
(2)
The map $\tilde{\fg}\to \fg$ is clearly a homomorphism of Lie algebras, 
and maps $\tilde{\fm}$ and $\tilde{\fh}$ to 
$\fm$ and $\fh$, respectively. 
A straightforward calculation shows that 
the zero section $\fg\to \tilde{\fg}$ 
is a homomorphism of Lie algebras 
and that it maps $\fm$ and $\fh$ 
into $\tilde{\fm}$ and $\tilde{\fh}$, respectively. 

By Proposition \ref{prop_ly_from_rt}, 
the induced maps $\tilde{\fm}\to\fm$ 
and $\fm\to\tilde{\fm}$ are homomorphisms 
of Lie-Yamaguti algebras. 

\smallbreak
(3)
It is again well known (and easy to check) that $0\oplus M$ is an abelian ideal of $\tilde{\fg}$. 
Since $0\oplus M_{\mathrm{n}}$ is the kernel of the homomorphism $\tilde{\fm}\to \fm$, 
it is an ideal. 
If $m, m'\in M_{\mathrm{n}}$, then in $\tilde{\fg}=\fg\oplus M$, 
\[
[(0, m), (0, m')]=(0, 0), 
\]
hence $(0, m) * (0, m') = 0$. 
For $(x, m)\in \tilde{\fm}$ and $m', m''\in M_{\mathrm{n}}$, we have
\[
[(x, m), (0, m')]=(0, xm'), 
\]
and therefore 
\[
[(x, m), (0, m'), (0, m'')] = [(0, (xm')_{\mathrm{s}}), (0, m'')] = 0. 
\]
By Proposition \ref{prop_ly_rep_ext}, 
$M_{\mathrm{n}}$ is a representation of $\fm$. 
We obtain (\ref{eq_rrt_to_rly})
by comparing (\ref{eq_rep_cor}) and (\ref{eq_ly_on_ext}). 
\end{proof}

\begin{proposition}\label{prop_fn_rly}
The assignment $(M, M_{\mathrm{n}}, M_{\mathrm{s}})\mapsto M_{\mathrm{n}}$ defines a functor 
\[
\mathcal{RLY}_{(\fg, \fm, \fh)}: \mathrm{Rep}(\fg, \fm, \fh)
\to \mathrm{Rep}(\fm).  
\]
\end{proposition}
\begin{proof}
If $f: (M, M_{\mathrm{n}}, M_{\mathrm{s}})\to (M', M'_{\mathrm{n}}, M'_{\mathrm{s}})$ is a homomorphism of 
representations of $(\fg, \fm, \fh)$, 
then by definition, 
$f: M\to M'$ is a homomorphism of representations of $\fg$ 
that commutes with the projections to the $\mathrm{n}$-factor and the $\mathrm{s}$-factor. 

Using (\ref{eq_rrt_to_rly}), it is straightforward to check that 
$f|_{M_{\mathrm{n}}}: M_{\mathrm{n}}\to M'_{\mathrm{n}}$ 
is a homomorphism of representations of $\fm$. 
\end{proof}

For local regular $s$-pairs and infinitesimal $s$-manifolds, we proceed in a similar way. 

\begin{definition}
Let $(\fg, \varphi)$ be a local regular $s$-pair. 
We define a \emph{representation} of  $(\fg, \varphi)$ to be 
a quadruple $(M, M_{\mathrm{n}}, M_{\mathrm{s}}, \tilde{\psi})$ 
where $(M, M_{\mathrm{n}}, M_{\mathrm{s}})$ is a representation of 
the reductive triple $(\fg, \fm, \fh)$ associated to $(\fg, \varphi)$, 
and $\tilde{\psi}: M\to M$ is an invertible linear map 
satisfying $\tilde{\psi}(M_{\mathrm{n}})\subseteq M_{\mathrm{n}}$, 
$\tilde{\psi}|_{M_{\mathrm{s}}} = \mathrm{id}_{M_{\mathrm{s}}}$ 
and $\varphi(x)\tilde{\psi}(m)=\tilde{\psi}(xm)$ 
for any $x\in\fg$ and $m\in M$. 

A representation $(M, M_{\mathrm{n}}, M_{\mathrm{s}}, \tilde{\psi})$ is called 
\emph{effective} or \emph{minimal} if $(M, M_{\mathrm{n}}, M_{\mathrm{s}})$ 
is effective or minimal, respectively, as a representation of $(\fg, \fm, \fh)$. 

A homomorphism $f: (M, M_{\mathrm{n}}, M_{\mathrm{s}}, \tilde{\psi})\to(M', M'_{\mathrm{n}}, M'_{\mathrm{s}}, \tilde{\psi}')$ 
of representations of $(\fg, \varphi)$ 
is a homomorphism of representations of $(\fg, \fm, \fh)$ 
satisfying $f\circ \tilde{\psi}=\tilde{\psi}'\circ f$. 

We write the categories of representations (resp. effective representations, 
minimal representations, effective and minimal representations) 
as $\mathrm{Rep}(\fg, \varphi)$ 
(resp. $\mathrm{Rep}^{\mathrm{e}}(\fg, \varphi)$, 
$\mathrm{Rep}^{\mathrm{m}}(\fg, \varphi)$, 
$\mathrm{Rep}^{\mathrm{em}}(\fg, \varphi)$). 
\end{definition}

\begin{proposition}
Let $(\fg, \varphi)$ be a local regular $s$-pair with the associated triple $(\fg, \fm, \fh)$. 
The assignment 
$(M, M_{\mathrm{n}}, M_{\mathrm{s}}, \tilde{\psi})\mapsto (M_{\mathrm{n}}, \tilde{\psi}|_{M_{\mathrm{n}}})$ 
defines a functor 
\[
\mathcal{RISM}_{(\fg, \varphi)}: \mathrm{Rep}(\fg, \varphi)
\to \mathrm{Rep}(\fm, \varphi|_{\fm}).  
\]
\end{proposition}
\begin{proof}
Let 
$(M, M_{\mathrm{n}}, M_{\mathrm{s}}, \tilde{\psi})$ 
be a representation of $(\fg, \varphi)$. 
We have already shown that $M_{\mathrm{n}}$ has a natural structure 
of a representation of $(\fg, \fm, \fh)$. 

Let us write $\sigma=\varphi|_\fm$, $\psi=\tilde{\psi}|_{M_{\mathrm{n}}}$ 
and check the conditions (RISM1--3). 

For $x\in \fm$ and $m\in M_{\mathrm{n}}$, 
using (\ref{eq_rrt_to_rly})
we have 
\begin{align*}
\rho(\sigma(x))\circ\psi(m)
& =(\sigma(x)(\psi(m)))_{\mathrm{n}} = (\varphi(x)(\tilde{\psi}(m)))_{\mathrm{n}} \\
&= (\tilde{\psi}(xm))_{\mathrm{n}} = \tilde{\psi}((xm)_{\mathrm{n}})
= \psi(\rho(x)m), 
\end{align*}
so (RISM1) holds. 
For $x, y\in \fm$ and $m\in M_{\mathrm{n}}$, 
\begin{align*}
\theta(x, \sigma(y))(m)
& = \sigma(y)\cdot(xm)_{\mathrm{s}} = \varphi(y)\tilde{\psi}((xm)_{\mathrm{s}}) \\
& = \tilde{\psi}(y\cdot(xm)_{\mathrm{s}}) = \psi(\theta(x, y)m), 
\end{align*}
and 
\begin{align*}
\theta(\sigma(x), y)\psi(m)
& = y\cdot(\sigma(x)\psi(m))_{\mathrm{s}} = y\cdot (\varphi(x)\tilde{\psi}(m))_{\mathrm{s}} = y\cdot (\tilde{\psi}(xm))_{\mathrm{s}} \\
& = y\cdot \tilde{\psi}((xm)_{\mathrm{s}}) = y\cdot (xm)_{\mathrm{s}} = \theta(x, y)m, 
\end{align*}
so (RISM2) holds. 
Finally, we have 
\begin{equation*}
\delta(x, y)\psi(m)
= [x, y]_\fh \psi(m) = \varphi([x, y]_\fh)\tilde{\psi}(m) 
= \tilde{\psi}([x, y]_\fh m) = \psi(\delta(x, y)m), 
\end{equation*}
and (RISM3) holds. 

Let 
$f: (M, M_{\mathrm{n}}, M_{\mathrm{s}}, \tilde{\psi})\to(M', M'_{\mathrm{n}}, M'_{\mathrm{s}}, \tilde{\psi}')$ 
be a homomorphism of representations of $(\fg, \varphi)$. 
In Proposition \ref{prop_fn_rly}, 
we saw that $f|_{M_{\mathrm{n}}}$ is a homomorphism of 
representations of $(\fg, \fm, \fh)$. 
From $f\circ \tilde{\psi}=\tilde{\psi}'\circ f$ 
it follows that 
$f|_{M_{\mathrm{n}}} \circ \tilde{\psi}|_{M_{\mathrm{n}}}=
\tilde{\psi}'|_{M'_{\mathrm{n}}}\circ f|_{M_{\mathrm{n}}}$, 
so $f|_{M_{\mathrm{n}}}$ is a homomorphism of 
representations of $(\fm, \sigma)$. 
\end{proof}

\section{Tight representations}

In this section, we consider the following situation: 
$\tilde{T}$ and $T$ are Lie-Yamaguti algebras, 
$\pi: \tilde{T}\to T$ is a surjective homomorphism of Lie-Yamaguti algebras 
and $\iota: T\to \tilde{T}$ is a section of $\pi$, i.e. 
a homomorphism of Lie-Yamaguti algebras such that $\pi\circ\iota=\mathrm{id}_T$. 

A \emph{surjective} homomorphism of Lie-Yamaguti algebras 
induces a (surjective) homomorphism of the standard enveloping Lie algebras. 

\begin{lemma}[{See e.g.\ \cite[Proposition 2.30]{Takahashi2021}}]
Let $T$ and $\tilde{T}$ be Lie-Yamaguti algebras 
and $\pi: \tilde{T}\to T$ a surjective homomorphism of Lie-Yamaguti algebras. 

Then there is a unique Lie algebra homomorphism $L(\pi): L(\tilde{T})\to L(T)$ 
such that $L(\pi)|_{\tilde{T}}=(\pi, 0)$, 
and it is a surjective homomorphism of reductive triples. 

For $\tilde{x}, \tilde{y}\in \tilde{T}$, 
$L(\pi)(D_{\tilde{x}, \tilde{y}})=D_{\pi(\tilde{x}), \pi(\tilde{y})}$ holds. 
If $\tilde{u}\in\ider(\tilde{T})$ and 
$u=L(\pi)(\tilde{u})$, then $u\circ\pi=\pi\circ\tilde{u}$ holds. 
\end{lemma}

This applies to $\pi$, but not to $\iota$. 
Now we consider a condition 
which guarantees the existence of a natural homomorphism $L(T)\to L(\tilde{T})$. 

\begin{notation}\label{notation_l_pi_i}
Let $T$ and $\tilde{T}$ be Lie-Yamaguti algebras, 
$\pi: \tilde{T}\to T$ a surjective homomorphism 
and $\iota: T\to \tilde{T}$ a homomorphism 
such that $\pi\circ\iota=\mathrm{id}_T$. 

Then $L(\pi, \iota): L_{L(\tilde{T})}(\iota(T))\to L(T)$ denotes the restriction 
of $L(\pi): L(\tilde{T})\to L(T)$ to $L_{L(\tilde{T})}(\iota(T))$. 
\end{notation}

\begin{definition}
A pair of Lie-Yamaguti algebra homomorphisms $\pi: \tilde{T}\to T$ 
and $\iota: T\to \tilde{T}$ with $\pi\circ\iota=\mathrm{id}_T$ is called \emph{tight} 
if the induced homomorphism $L(\pi, \iota): L_{L(\tilde{T})}(\iota(T))\to L(T)$ 
is an isomorphism. 

A representation $V$ of a Lie-Yamaguti algebra $T$ is called \emph{tight} 
if the pair of the induced homomorphisms $T\oplus V\to T$ 
and $T\to T\oplus V$ is tight. 
We denote the full subcategory of $\mathrm{Rep}(T)$ 
consisting of tight representations by $\mathrm{Rep}^{\mathrm{t}}(T)$.

We say a representation of an infinitesimal $s$-manifold $(T, \sigma)$ is \emph{tight} 
if it is tight as a representation of the Lie-Yamaguti algebra $T$, 
and denote the subcategory of $\mathrm{Rep}(T, \sigma)$ 
consisting of tight representations by $\mathrm{Rep}^{\mathrm{t}}(T, \sigma)$. 
\end{definition}

\begin{lemma}\label{lem_tight}
Let $(V, \rho, \theta, \delta)$ be a representation of a Lie-Yamaguti algebra $T$. 
Then the following are equivalent. 
\begin{enumerate}[(a)]
\item
The representation $V$ is tight. 
\item
If $\sum_{i=1}^k D_{x_i, y_i}=0$ for $x_i, y_i\in T$, 
then $\sum_{i=1}^k \delta(x_i, y_i) = 0$. 
\item
There is a linear map $\ider(T)\to \mathrm{End}(V)$ 
which maps $D_{x, y}$ to $\delta(x, y)$ for any $x, y\in T$. 
\end{enumerate}
\end{lemma}
\begin{proof}
It is obvious that (b) and (c) are equivalent. 

An element $\tilde{u}$ of the kernel of 
$L_{L(T\oplus V)}(T\oplus 0)=(T\oplus 0)\oplus \ider_{T\oplus V}(T\oplus 0)\to L(T)$ 
is obviously in $\ider_{T\oplus V}(T\oplus 0)$. 
Let $\tilde{u}=\sum D_{(x_i, 0), (y_i, 0)}$ be an element of $\ider_{T\oplus V}(T\oplus 0)$. 
We have 
\begin{align*}
\tilde{u}(z, v) & =
\sum D_{(x_i, 0), (y_i, 0)}(z, v)=\sum [(x_i, 0), (y_i, 0), (z, v)] \\
& =\left(\left(\sum D_{x_i, y_i}\right) z, \left(\sum \delta(x_i, y_i)\right)v\right), 
\end{align*}
and $\tilde{u}$ is in the kernel if and only if $\sum D_{x_i, y_i}=0$. 
Thus, the kernel is $0$ if and only if (b) holds. 
\end{proof}

\begin{proposition}
(1)
Let $T$ be a Lie-Yamaguti algebra 
and $(V, \rho, \theta, \delta)$ a tight representation of $T$. 
If $(W, \rho', \theta', \delta')$ is a subrepresentation or a quotient representation of $V$, 
then $W$ is tight. 

Hence, $\mathrm{Rep}^{\mathrm{t}}(T)$  
is an abelian subcategory of $\mathrm{Rep}(T)$. 

(2)
Let $(T, \sigma)$ be an infinitesimal $s$-manifold 
and $(V, \rho, \theta, \delta, \psi)$ a tight representation of $(T, \sigma)$. 
If $(W, \rho', \theta', \delta', \psi')$ is a subrepresentation or a quotient representation of $V$, 
then $W$ is tight. 

Hence, $\mathrm{Rep}^{\mathrm{t}}(T, \sigma)$  
is an abelian subcategory of $\mathrm{Rep}(T, \sigma)$. 
\end{proposition}
\begin{proof}
(1)
For any $x, y\in T$, 
$\delta'(x, y)$ is induced from $\delta(x, y)$, 
and so the condition Lemma \ref{lem_tight} (b) for $V$ 
implies the same condition for $W$. 

(2)
Obvious from (1). 
\end{proof}

\begin{example}\label{ex_non_tight}
Let us give an example of a representation which is not tight. 
Let $\fg$ be a Lie algebra 
and $T=\fg_{1, 0}$ with the notation of Example \ref{ex_l_to_ly}. 
This is $\fg$ regarded as a Lie-Yamaguti algebra with the trivial triple product, 
and $L(T)=\fg$ as a Lie algebra. 

Let $V$ be a $1$-dimensional space and 
$\lambda: \fg\to\mathbb{K}$ a linear function. 
We define a representation of $T$ on $V$ 
by $\rho(x)(v)=\lambda(x)v$, $\theta(x, y)(v)=0$ and $\delta(x, y)(v)=-\lambda([x, y])v$. 

The induced operations on $\tilde{T}:=T\oplus V$ are given by 
\begin{align*}
(x, u)*(y, v) & =([x, y], \lambda(x)v-\lambda(y)u), \\
[(x, u), (y, v), (z, w)] & =(0, -\lambda([x, y])w), 
\end{align*}
and $D_{(x, u), (y, v)}=-\lambda([x, y])p_V$ where $p_V: T\oplus V\to V$ is the projection. 
Let $\iota: T\to\tilde{T}$ be the zero section. 
Assuming that $\lambda([\fg, \fg])\not=0$ (e.g., $[\fg, \fg]=\fg$ and $\lambda\not=0$), 
$L_{L(\tilde{T})}(\iota(T))$ is $\iota(T)\oplus \mathbb{K}p_V$, 
while $L(T)=T$, 
and hence this representation is not tight. 

This construction can be generalized as follows. 
We take a Lie algebra homomorphism 
$\bar{\lambda}: \fg\to\fa$, 
a central extension $\tilde{\fa}\to\fa$, 
a representation $(V, \alpha)$ of $\tilde{\fa}$, 
and a linear lift $\lambda: \fg\to\tilde{\fa}$ of $\bar{\lambda}$. 
Then $\rho(x)=\alpha(\lambda(x))$, $\theta(x, y)\equiv 0$ 
and $\delta(x, y):=\alpha([\lambda(x), \lambda(y)]-\lambda([x, y]))$ 
defines a representation of $\fg_{1, 0}$. 
\end{example}

\begin{example}
Let us see that an extension of tight representations is not necessarily tight. 
For $\fg$, $T$ and $\lambda: \fg\to\mathbb{K}$ 
as in the previous example, 
let $V=\mathbb{K}^2$  
and $\rho(x)=\begin{pmatrix} 0 & \lambda(x) \\ 0 & 0\end{pmatrix}$, 
$\theta\equiv 0$ and 
$\delta(x, y)=\begin{pmatrix} 0 & -\lambda([x, y]) \\ 0 & 0\end{pmatrix}$. 
This can be seen as the case where $\tilde{\fa}=\mathbb{K}$, $\fa=0$ 
and the representation $\alpha$ of $\tilde{\fa}$ on $V$ is given by 
$c\mapsto \begin{pmatrix} 0 & c \\ 0 & 0\end{pmatrix}$. 

Then $V$ is an extension of the trivial module $\mathbb{K}$ by $\mathbb{K}$, 
and while $\mathbb{K}$ is tight, $V$ is not if $\lambda([\fg, \fg])\not=0$. 
\end{example}

The representations associated to representations of a reductive triple 
(Proposition \ref{prop_fn_rly}) are tight under a certain condition 
on the triple. 

\begin{proposition}
Let $(\fg, \fm, \fh)$ be a reductive triple and 
assume that the action of $[\fm, \fm]_{\fh}$ on $\fm$ 
is effective. 
Then, for any representation $(M, M_{\mathrm{n}}, M_{\mathrm{s}})$ of $(\fg, \fm, \fh)$, 
the representation $M_{\mathrm{n}}$ of $\fm$ is tight. 
Consequently, 
we can restrict the codomain of the functor 
$\mathcal{RLY}_{(\fg, \fm, \fh)}: \mathrm{Rep}(\fg, \fm, \fh)
\to \mathrm{Rep}(\fm)$ 
to obtain a functor 
\[
\mathrm{Rep}(\fg, \fm, \fh)
\to \mathrm{Rep}^{\mathrm{t}}(\fm), 
\]
which we again denote by $\mathcal{RLY}_{(\fg, \fm, \fh)}$. 
\end{proposition}
\begin{proof}
We use the criterion (b) in Lemma \ref{lem_tight}. 
Assume $\sum D_{x_i, y_i}=0$ for $x_i, y_i\in\fm$. 
Then, for any $z\in \fm$, 
$\sum D_{x_i, y_i}(z)=[\sum [x_i, y_i]_{\fh}, z]=0$. 
By the assumption, $\sum [x_i, y_i]_{\fh}=0$. 
Thus, for any $m\in M_{\mathrm{n}}$, 
we have $\sum \delta(x_i, y_i)(m)=(\sum [x_i, y_i]_{\fh})m=0$ 
by (\ref{eq_rrt_to_rly}). 
\end{proof}

This applies in particular to the triples of the form $(L(T), T, \ider(T))$. 
If $\fg$ is simple and $\fm\not=0$, 
then $(\fg, \fm, \fh)\cong (L(\fm), \fm, \ider(\fm))$ by Lemma \ref{lem_semisimple} (2) 
and the proposition applies. 
(If $\fm=0$, then $[\fm, \fm]_\fh=0$ acts effectively on $\fm$ again, 
although this case is not so interesting.)

\medbreak
To study the tightness of representations, 
we take a closer look at the homomorphisms $L(\pi, \iota)$ 
(see Notation \ref{notation_l_pi_i}). 

\begin{lemma}\label{lem_l_pi_iota}
Let $T$ and $\tilde{T}$ be Lie-Yamaguti algebras, 
$\pi: \tilde{T}\to T$ a surjective homomorphism 
and $\iota: T\to \tilde{T}$ a homomorphism 
such that $\pi\circ\iota=\mathrm{id}_T$. 
Then the following hold. 

(1) 
$L(\pi, \iota)$ is surjective. 

(2)
The kernel $\mathrm{Ker}\ L(\pi, \iota)$ is contained 
in $\ider_{\tilde{T}}(\iota(T))$, and also 
in the center of $L_{L(\tilde{T})}(\iota(T))$. 
\end{lemma}
\begin{proof}
Let $T'=\iota(T)$. 

(1)
From $\pi\circ\iota=\mathrm{id}_T$ it follows that
$L(\pi, \iota)(T')=\pi(T')=\pi(\iota(T))=T$, 
and consequently that 
\[
\mathrm{Im}\ L(\pi, \iota) = L(\pi, \iota)(T'+[T', T']) =
T+[T, T]=L(T). 
\]

(2)
Since $L(\pi, \iota)$ restricts to $T'\overset\sim\to T$, 
an element of $\mathrm{Ker}\ L(\pi, \iota)$ can be written as $(0, \tilde{u})$, 
$\tilde{u}\in\ider_{\tilde{T}}(T')$, 
i.e., the first statement holds. 

Since $\tilde{u}: \tilde{T}\to\tilde{T}$ induces the zero map on $T$, 
we have $\tilde{u}(T')\subseteq \tilde{u}(\tilde{T})\subseteq\mathrm{Ker}\ \pi$. 
We also see $\tilde{u}(T')\subseteq T'$ 
from the fact that $T'$ is closed under the triple product. 
Thus we have $\tilde{u}(T')=0$, 
or in terms of the Lie bracket on $L(\tilde{T})$, 
$\tilde{u}$ commutes with $T'$. 
It also commutes with $[T', T']$ by Jacobi identity, 
and so belongs to the center of $L_{L(\tilde{T})}(T')$. 
\end{proof}

\begin{corollary}\label{cor_ss_section}
Let $T$ be an $L$-semisimple Lie-Yamaguti algebra, 
$\tilde{T}$ a Lie-Yamaguti algebra and 
$\pi: \tilde{T}\to T$ and $\iota: T\to \tilde{T}$ homomorphisms of Lie-Yamaguti algebras 
such that $\pi\circ\iota=\mathrm{id}_T$. 

Then $L(\pi, \iota)$ admits a unique section $f: L(T)\to L_{L(\tilde{T})}(\iota(T))$, 
i.e., a Lie algebra homomorphism $f$ such that $L(\pi, \iota)\circ f=\mathrm{id}_{L(T)}$. 
In other words, 
there is a unique Lie subalgebra $\fg$ of $L_{L(\tilde{T})}(\iota(T))$ 
such that $L(\pi, \iota)|_{\fg}$ is an isomorphism. 
Specifically, $\fg=[L_{L(\tilde{T})}(\iota(T)), L_{L(\tilde{T})}(\iota(T))]$ holds. 
\end{corollary}

Note that 
$f$ and $\fg$ are not necessarilly compatible with the reductive decomposition. 
See Example \ref{ex_not_compat_with_decomp}, 
and also Lemma \ref{lem_tightness}. 

\begin{proof}
Since $\mathrm{Ker}\ L(\pi, \iota)$ is abelian 
and $L(T)$ is semisimple, 
$\mathrm{Ker}\ L(\pi, \iota)$ is the radical of $L_{L(\tilde{T})}(\iota(T))$, 
and the existence of a section follows from Levi's theorem. 
We have a decomposition $L_{L(\tilde{T})}(\iota(T))=\fg\oplus \mathrm{Ker}\ L(\pi, \iota)$ 
as a vector space, 
where $\fg$ is a subalgebra of $L_{L(\tilde{T})}(\iota(T))$ isomorphic to $L(T)$, 
and by $[L(T), L(T)]=L(T)$ and Lemma \ref{lem_l_pi_iota} (2), 
we have $\fg=[\fg, \fg]=[L_{L(\tilde{T})}(\iota(T)), L_{L(\tilde{T})}(\iota(T))]$. 
\end{proof}

Now, assuming that $L(\pi, \iota)$ admits a section $f$, 
we study when $(\pi, \iota)$ is tight, i.e., 
$L_{L(\tilde{T})}(\iota(T))$ itself is isomorphic to $L(T)$.

First, we give a number of equivalent conditions. 

\begin{lemma}\label{lem_tightness}
Assume that 
$\pi: \tilde{T}\to T$ and 
$\iota: T\to \tilde{T}$ are homomorphisms of Lie-Yamaguti algebras 
satisfying $\pi\circ\iota = \mathrm{id}_T$, 
and that 
$f: L(T)\to L_{L(\tilde{T})}(\iota(T))$ is a Lie algebra homomorphism 
such that $L(\pi, \iota)\circ f = \mathrm{id}_{L(T)}$. 
Then the following are equivalent. 
\begin{enumerate}[(a)]
\item
$(\pi, \iota)$ is tight, i.e., 
$L(\pi, \iota)$ is an isomorphism. 
\item
$f$ is an isomorphism. 
\item
$f(L(T))$ is compatible with the reductive decomposition. 
\item
$f(T)\supseteq \iota(T)$. 
\item
$f(L(T))\supseteq \iota(T)$. 
\item
$f(T)\subseteq \iota(T)$. 
\end{enumerate}
\end{lemma}
\begin{proof}
Let $T'=\iota(T)$ and $\fg=f(L(T))$. 

The equivalence of (a) and (b) is obvious. 
If (b) holds, 
then from $L(\pi, \iota)(T')=T$ and $L(\pi, \iota)(\ider_{\tilde{T}}(T'))=\ider(T)$ 
we have $f(T)=T'$ and $f(\ider(T))=\ider_{\tilde{T}}(T')$, 
and (c), (d) and (f) hold. 

Assume that (c) holds, i.e., 
$\fg=(\fg\cap T')\oplus (\fg\cap \ider_{L(\tilde{T})}(T'))$. 
Then $L(T)=\pi(\fg\cap T')\oplus L(\pi, \iota)(\fg\cap \ider_{L(\tilde{T})}(T'))$, 
and it follows that $\pi(\fg\cap T')=T$. 
Since $\pi|_{T'}: T'\to T$ is an isomorphism, (e) holds. 

If (f) holds, for any $x\in T$ we have $f(x)=\iota(y)$ for some $y\in T$, 
and sending by $L(\pi, \iota)$ we see that $x=y$. 
Thus $f|_T=\iota$, and (d) holds. 

The implication (d) $\Rightarrow$ (e) is obvious. 

Assume (e), i.e., $\fg\supseteq T'$. 
Then $L_{L(\tilde{T})}(T')=T'+[T', T']\subseteq \fg$. 
By Lemma \ref{lem_l_pi_iota} (1), $L_{L(\tilde{T})}(T')=\fg$ 
and $f$ is an isomorphism, so (e) $\Rightarrow$ (b) holds. 
\end{proof}

We give two sufficient conditions for the tightness. 
\begin{theorem}\label{thm_tightness}
Assume that 
$\pi: \tilde{T}\to T$ and 
$\iota: T\to \tilde{T}$ are homomorphisms of Lie-Yamaguti algebras 
satisfying $\pi\circ\iota = \mathrm{id}_T$, 
and that 
$f: L(T)\to L_{L(\tilde{T})}(\iota(T))$ is a Lie algebra homomorphism 
such that $L(\pi, \iota)\circ f = \mathrm{id}_{L(T)}$. 

(1)
If $[T, T, T]=T$, or equivalently $[\ider(T), T]=T$, 
then $(\pi, \iota)$ is tight. 

(2)
Assume that $(T, \sigma)$ and $(\tilde{T}, \tilde{\sigma})$ are infinitesimal $s$-manifolds, 
that $\pi$ and $\iota$ are homomorphisms of infinitesimal $s$-manifolds 
and that $[L(T), L(T)]=L(T)$. 
Then $(\pi, \iota)$ is tight. 

We may also drop the regularity of $\tilde{T}$: 
$(T, \sigma)$ is an infinitesimal $s$-manifold, 
$\tilde{T}$ is a Lie-Yamaguti algebra 
and $\tilde{\sigma}$ is an automorphism of $\tilde{T}$ 
satisfying $D_{\tilde{\sigma}(\tilde{x}), \tilde{\sigma}(\tilde{y})}=D_{\tilde{x}, \tilde{y}}$, 
and $\pi$ and $\iota$ are compatible with $\sigma$ and $\tilde{\sigma}$. 
\end{theorem}

\begin{remark}
The condition in (1) does not necessarily hold even for $L$-simple Lie-Yamaguti algebras: 
let $T$ be a simple Lie algebra regarded as a Lie-Yamaguti algebra 
with the trivial triple product. 
Then $\ider(T)=0$ and $L(T)=T$ is a simple Lie algebra, 
but $[T, T, T]=0$. 
\end{remark}

\begin{proof}
Let $T'=\iota(T)$ and $\fg=f(L(T))$. 

\smallbreak
(1)
From $L(\pi, \iota)\circ f=\mathrm{id}_{L(T)}$, 
we see that $L(\pi, \iota)(f((x, 0))-(\iota(x), 0))=0$ for any $x\in T$ 
and hence that $f(T)\subseteq T'\oplus \mathrm{Ker}\, L(\pi, \iota)$. 
On the other hand, from $L(\pi, \iota)(f((0, D_{x, y})))=(0, D_{x, y})$ 
we have 
$f(\ider(T))\subseteq L(\pi, \iota)^{-1}(\ider(T)) = \ider_{\tilde{T}}(T')$. 
Using the assumption and Lemma \ref{lem_l_pi_iota} (2), we have 
\[
f(T)=f([\ider(T), T]) 
= [f(\ider(T)), f(T)]
\subseteq [\ider_{\tilde{T}}(T'), T'\oplus \mathrm{Ker}\, L(\pi, \iota)]
\subseteq T', 
\]
and we are done by the previous lemma. 

\smallbreak
(2)
Let $L(\sigma)$ and $L(\tilde{\sigma})$ denote the automorphisms of $L(T)$ and $L(\tilde{T})$ 
induced by $\sigma$ and $\tilde{\sigma}$, respectively. 
By the previous lemma, it suffices to show that $T'\subseteq \fg$. 

From the compatibility of $\iota$ with $\sigma$ and $\tilde{\sigma}$, 
it follows that $\tilde{\sigma}(T')=T'$ and therefore that $L(\tilde{\sigma})$ 
restricts to an automorphism of $L_{L(\tilde{T})}(T')$. 
In the decomposition $L_{L(\tilde{T})}(T')=\fg\oplus \mathrm{Ker}\, L(\pi, \iota)$ as a vector space, 
$\mathrm{Ker}\, L(\pi, \iota)$ is in the center of $L_{L(\tilde{T})}(T')$ by Lemma \ref{lem_l_pi_iota} (2) 
and 
\[
[L_{L(\tilde{T})}(T'), L_{L(\tilde{T})}(T')]=[\fg, \fg]=\fg
\]
holds by the assumption $[L(T), L(T)]=L(T)$. 

It follows that $\fg$ is invariant under $L(\tilde{\sigma})$. 
Now, given any $x\in T$, write 
$(\tilde{x}, \tilde{u})=f((x, 0))\in \fg$. 
From $L(\pi, \iota)\circ f=\mathrm{id}_{L(T)}$, it follows that $\pi(\tilde{x})=x$, 
and combining with $\tilde{x}\in T'$ we have $\tilde{x}=\iota(x)$. 
From  
$L(\tilde{\sigma})(D_{\tilde{x}, \tilde{y}})=D_{\tilde{\sigma}(\tilde{x}), \tilde{\sigma}(\tilde{y})}=D_{\tilde{x}, \tilde{y}}$, 
$L(\tilde{\sigma})$ fixes any element of $\ider(\tilde{T})$, and we see that 
\[
(L(\tilde{\sigma})-1)(\tilde{x}, \tilde{u})=((\tilde{\sigma}-1)(\tilde{x}), 0)
= (\tilde{\sigma}(\iota(x))-\iota(x), 0)
= (\iota((\sigma-1)(x)), 0), 
\]
and this belongs to $\fg$. 
Since $\sigma-1$ is invertible, we see that $T'=\iota(T)\subseteq\fg$. 
\end{proof}

Here are a few examples 
concerning Theorem \ref{thm_tightness} (1). 

\begin{example}
(1)
If $\fg=\mathfrak{sl}_n$, 
$\fm\subseteq\fg$ is the set of matrices whose diagonal components are $0$ 
and $\fh\subseteq\fg$ is the set of diagonal matrices, 
then $(\fg, \fm, \fh)$ is a reductive triple, 
$L(\fm)\cong\fg$ 
and $[\fh, \fm]=\fm$ holds. 

(2)
For a semisimple Lie algebra $\fg$, 
let $T=\fg_{1, t}$ with $t\not=0, -1/4$ 
from Example \ref{ex_l_to_ly}. 
Then $L(T)$ is isomorphic to $\fg\oplus\fg$, 
so is semisimple. 
$\ider(T)$ is isomorphic to $\fg$ 
and its action on $T$ is isomorphic to the adjoint action. 
Thus $[T, T, T]=T$ holds. 
\end{example}

\begin{example}\label{ex_not_compat_with_decomp}
The $L$-(semi)simplicity of $T$ is not enough for the tightness 
even in a situation coming from a representation of $T$ on $V$, 
as we saw in Example \ref{ex_non_tight}. 

On the other hand, 
if $L(T)$ is semisimple, 
there is a section $L(T)\to L_{L(\tilde{T})}(\iota(T))$ given by 
the commutator $[L_{L(\tilde{T})}(\iota(T)), L_{L(\tilde{T})}(\iota(T))]$ 
by Corollary \ref{cor_ss_section}. 
Let us check that the conditions of Lemma \ref{lem_tightness} 
do not hold in Example \ref{ex_non_tight}. 
From 
\[
[((x, 0), c), ((y, 0), d)]=(([x, y], 0), -\lambda([x, y])p_V), 
\]
we see that the lift of $L(T)$ is $\{((x, 0), -\lambda(x)p_V)\mid x\in T\}$, 
and this is not compatible 
with the decomposition $L(\tilde{T})=\tilde{T}\oplus \ider(\tilde{T})$ 
if $\lambda\not=0$. 
\end{example}

\section{Correspondence of representations}

In this final section, we construct a functor from 
the category of tight representations of a Lie-Yamaguti algebra 
(resp. an infinitesimal $s$-manifold) 
to the category of effective minimal representations of 
the associated reductive triple (resp. local regular $s$-pair), 
and show that these categories are equivalent. 
In particular, we obtain our main results in the $L$-semisimple case. 

In what follows, for a representation $V$ of $T$, 
we will sometimes write the elements $(x, 0), (0, v)$ of $\tilde{T}=T\oplus V$ 
simply as $x, v$. 

\begin{proposition}\label{prop_rly_to_rrt}
Let $T$ be a Lie-Yamaguti algebra, 
$V$ a representation of $T$, 
and $\tilde{T}:=T\oplus V$ the extension associated to $V$. 
Let $\pi: \tilde{T}\to T$ be the projection, $\iota: T\to \tilde{T}$ the zero section, 
and identify $T$ with $\iota(T)$. 
Assume furthermore that the representation $V$ is tight, i.e., 
$L(\pi, \iota): L_{L(\tilde{T})}(T)\to L(T)$ is an isomorphism. 

(1)
The ideal $\mathrm{Ker}\, L(\pi)$ is abelian 
and is equal to $V\oplus \ider_{\tilde{T}}(T, V)$. 
One can regard $\ider_{\tilde{T}}(T, V)$ as a subspace of 
$\mathrm{Hom}(T, V)$, and 
$(\mathrm{Ker}\, L(\pi), V, \ider_{\tilde{T}}(T, V))$
is an effective and minimal representation of the reductive triple $(L(T), T, \ider(T))$. 
Concretely, 
for $x, y, z\in T$ and $v\in V$, 
if we define $\theta_2(z, v)\in\mathrm{Hom}(T, V)\subseteq\mathrm{End}(T\oplus V)$ 
by $\theta_2(x, v)y=\theta(x, y)v$, 
then $D_{x, v}=-\theta_2(x, v)$ as an element of $\ider_{\tilde{T}}(T, V)$ and 
\begin{equation}\label{eq_tightrep_to_reptriple}
\begin{array}{rcl}
(x, 0)\cdot (v, 0) & = & 
(\rho(x)v, -\theta_2(x, v)), \\
(x, 0)\cdot (0, f) & = & (-f(x), 0), \quad\hbox{in particular} \\
(x, 0)\cdot (0, D_{z, v}) & = & (-\theta(z, x)v, 0), \\
(0, D_{x, y})\cdot (v, 0) & = & (\delta(x, y)v, 0), \\
(0, D_{x, y})\cdot (0, f) & = & (0, [D_{(x, 0), (y, 0)}, f]) 
= (0, \delta(x, y)\circ f - f\circ D_{x, y}), \\
 & & \qquad \hbox{and in particular} \\
(0, D_{x, y})\cdot (0, D_{z, v}) & = &
(0, \theta_2(z, v)\circ D_{x, y} - \delta(x, y)\circ \theta_2(z, v)). 
\end{array}
\end{equation}
A homomorphism $\chi: V\to V'$ of representations of $T$ 
induces a homomorphism 
\[
\chi\times \chi_*: V\oplus \ider_{\tilde{T}}(T, V)\to V'\oplus \ider_{\tilde{T}}(T, V'); 
(v, f)\mapsto (\chi(v), \chi\circ f)
\]
of representations of $(L(T), T, \ider(T))$. 

(2)
Assume that $(T, \sigma)$ is an infinitesimal $s$-manifold 
and that $(V, \psi)$ is a representation of $(T, \sigma)$. 
Let $\tilde{\sigma}=(\sigma, \psi)$. 
Then $L(\tilde{\sigma})$ preserves 
the decomposition $L(\tilde{T})=L_{L(\tilde{T})}(T)\oplus \mathrm{Ker}\ L(\pi)$, 
and acts as $(x, f)\mapsto (\sigma(x), f)$ 
on $L_{L(\tilde{T})}(T)$ and 
\begin{equation}\label{eq_tightrep_to_rep_spair}
(v, f)\mapsto (\psi(v), f)
\end{equation}
on $\mathrm{Ker}\, L(\pi)$. 
On $(\mathrm{Ker}\, L(\pi), V, \ider_{\tilde{T}}(T, V))$, 
(\ref{eq_tightrep_to_reptriple}) and (\ref{eq_tightrep_to_rep_spair}) define 
a structure of a representation of the local regular $s$-pair 
$(L(T), L(\sigma))$. 

If $\chi: (V, \psi)\to (V', \psi')$ is a homomorphism of representations of $(T, \sigma)$, 
then $\chi\times \chi_*$ is a homomorphism 
of representations of $(L(T), L(\sigma))$. 
\end{proposition}

\begin{proof}
(1)
Since 
$V$ is an abelian ideal of $\tilde{T}$ by Proposition \ref{prop_ly_rep_ext}, 
we have $\ider_{\tilde{T}}(V)=0$ and hence 
\[
L(\tilde{T})=\tilde{T}\oplus \ider(\tilde{T}) 
= (T\oplus V)\oplus (\ider_{\tilde{T}}(T) + \ider_{\tilde{T}}(T, V)). 
\]
Let $((x, v), h+h')$ be an element of $L(\tilde{T})$, 
where $x\in T$, $v\in V$, $h\in \ider_{\tilde{T}}(T)$ 
and $h'\in\ider_{\tilde{T}}(T, V)$. 
Let $\overline{h}$ and $\overline{h'}$ be the self-maps of $T$ 
induced by $h$ and $h'$, respectively. 
From $[\tilde{T}, V, \tilde{T}]\subseteq V$, we have $h'(y)\in V$ for any $y\in T$, 
and hence $\overline{h'}=0$. 
Thus
\[
L(\pi)(((x, v), h+h'))=(x, \overline{h}+\overline{h'})=(x, \overline{h}). 
\]
From this
we see that $V\oplus \ider_{\tilde{T}}(T, V)\subseteq \mathrm{Ker}\, L(\pi)$ holds. 

Conversely, if $((x, v), h+h')\in\mathrm{Ker}\, L(\pi)$, 
then we have $x=0$ and $\overline{h}=0$. 
By the assumption that $L(\pi, \iota): T\oplus \ider_{\tilde{T}}(T)\to L(T)$ is an isomorphism, 
this implies $h=0$, and we have $\mathrm{Ker}\, L(\pi)=V\oplus \ider_{\tilde{T}}(T, V)$. 

Since $V$ is an abelian ideal, any element of $\ider_{\tilde{T}}(T, V)$ 
maps $\tilde{T}$ to $V$ and $V$ to $0$. 
Thus, 
via 
$\ider_{\tilde{T}}(T, V)\subseteq \mathrm{End}(\tilde{T})\cong 
\mathrm{Hom}(T, T)\oplus\mathrm{Hom}(T, V)\oplus\mathrm{Hom}(V, T)\oplus\mathrm{Hom}(V, V)$, 
$\ider_{\tilde{T}}(T, V)$ can be regarded as a subspace of $\mathrm{Hom}(T, V)$. 

From this and $D_{v, w}=0$ for $v, w\in V$, 
we see that the ideal $\mathrm{Ker}\, L(\pi)=V\oplus \ider_{\tilde{T}}(T, V)$ is abelian. 
As usual, restricting the Lie bracket to 
$L_{\tilde{T}}(T)\times \mathrm{Ker}\, L(\pi)$, 
we have a structure of a representation of $L_{\tilde{T}}(T)\cong L(T)$ 
on $\mathrm{Ker}\, L(\pi)$. 

For convenience, let us write 
$(\fg, \fm, \fh)=(L(T), T, \ider(T))$ and 
$(M, M_{\mathrm{n}}, M_{\mathrm{s}})=(\mathrm{Ker}\, L(\pi), V, \ider_{\tilde{T}}(T, V))$. 
Then from the fact that $(L(\tilde{T}), \tilde{T}, \ider(\tilde{T}))$ 
is a reductive triple, we have 
$\fm M_{\mathrm{s}}\subseteq M_{\mathrm{n}}$, 
$\fh M_{\mathrm{n}}\subseteq M_{\mathrm{n}}$ 
and $\fh M_{\mathrm{s}}\subseteq M_{\mathrm{s}}$, 
so we have a representation of the reductive triple. 
Furthermore, 
we identified $\ider_{\tilde{T}}(T, V)$ 
with a subspace of $\mathrm{Hom}(T, V)$ in a way that their actions on $T$ coincide. 
This implies that 
$M_{\mathrm{s}}=\ider_{\tilde{T}}(T, V)\hookrightarrow \mathrm{Hom}(T, V)=\mathrm{Hom}(\fm, M_{\mathrm{n}})$, 
i.e., the representation is effective. 
We have 
$(\fm M_{\mathrm{n}})_{\mathrm{s}}=(TV)_{\mathrm{s}}=\ider_{\tilde{T}}(T, V)=M_{\mathrm{s}}$, 
so the representation is minimal. 

The equalities (\ref{eq_tightrep_to_reptriple}) follow 
from (\ref{eq_rep_cor}) and (\ref{eq_lt}). 
First, in $\tilde{T}$, 
\[
[(x, 0), (0, v), (y, 0)]=(0, -\theta(x, y)v)=(0, -\theta_2(x, v)y), 
\]
so $D_{x, v}=-\theta_2(x, v)$ holds. 
Then, for example, 
in $L(\tilde{T})$ we have 
\[
[(x, 0), (v, 0)]=(x*v, D_{x, v})=(\rho(x)v, -\theta_2(x, v)), 
\]
hence $(x, 0)\cdot (v, 0)=(\rho(x)v, -\theta_2(x, v))$, and so on. 

If $\chi: V\to V'$ be a homomorphism of representation of $T$, 
it is straightforward to check that $\chi\times \chi_*$ 
is a homomorphism of representations of $L(T)$. 
We may also prove this 
by showing that it induces a homomorphism of Lie algebras 
$L(T\oplus V)\to L(T\oplus V')$. 

\smallbreak
(2)
For $x\in T$ 
we have 
$L(\tilde{\sigma})((x, 0))=(\tilde{\sigma}(x), 0)=(\sigma(x), 0)\in T$. 
Similarly, 
for $v\in V$ 
we have 
$L(\tilde{\sigma})((v, 0))=(\tilde{\sigma}(v), 0)=(\psi(v), 0)\in V$. 
By the proof of \cite[Proposition 4.5 (2)]{Takahashi2021}, 
$(\tilde{T}, \tilde{\sigma})$ satisfies (ISM1-3). 
This means that $D_{\tilde{\sigma}(\tilde{x}), \tilde{\sigma}(\tilde{y})}=D_{\tilde{x}, \tilde{y}}$ 
holds for any $\tilde{x}, \tilde{y}\in\tilde{T}$, 
and $L(\tilde{\sigma})$ acts trivially on $\ider(\tilde{T})$. 
Thus $L(\tilde{\sigma})$ preserves $L_{L(\tilde{T})}(T)$ and 
$\mathrm{Ker}\,L(\pi)=V\oplus\ider_{\tilde{T}}(T, V)$. 

Let $\varphi$ be the automorphism $L(\sigma)$ of $L(T)$, 
which one may identify with the restriction of $L(\tilde{\sigma})$ to $L_{L(\tilde{T})}(T)$,
and let $\tilde{\psi}$ be the restriction of $L(\tilde{\sigma})$ to 
$\mathrm{Ker}\,L(\pi)=V\oplus\ider_{\tilde{T}}(T, V)$. 
Concretely, the latter is given by $\tilde{\psi}(v, f)=(\psi(v), f)$. 
Thus we have to show that 
$\varphi(x)\cdot\tilde{\psi}(m)=\tilde{\psi}(x\cdot m)$. 
The left hand side is 
$[\varphi(x), \tilde{\psi}(m)]$ in $L(\tilde{T})$, 
which is equal to 
$[\tilde{\sigma}(x), \tilde{\sigma}(m)]=\tilde{\sigma}([x, m])=\tilde{\psi}(x\cdot m)$. 

If $\chi: (V, \psi)\to (V', \psi')$ is a homomorphism, 
then we have
\[
(\chi\times \chi_*)\tilde{\psi}(v, f)
= (\chi(\psi(v)), \chi\circ f)
= (\psi'(\chi(v)), \chi\circ f)=\tilde{\psi}'((\chi\times \chi_*)(v, f)). 
\]
\end{proof}

\begin{definition}
(1)
Let $T$ be a Lie-Yamaguti algebra. 
We define the functor 
\[
\mathcal{RRT}_T: \mathrm{Rep}^{\mathrm{t}}(T)\to \mathrm{Rep}^{\mathrm{em}}(L(T), T, \ider(T))
\]
from the category of tight representations of $T$ 
to the category of effective minimal representations of $(L(T), T, \ider(T))$ 
by 
\[
\mathcal{RRT}_T(V):=V\oplus \ider_{T\oplus V}(T, V), 
\]
where we regard the right hand side as a representation of $(L(T), T, \ider(T))$ 
by the previous proposition. 

(2)
Let $(T, \sigma)$ be an infinitesimal $s$-manifold. 
We define the functor 
\[
\mathcal{RLRS}_{(T, \sigma)}: \mathrm{Rep}^{\mathrm{t}}(T, \sigma)
\to \mathrm{Rep}^{\mathrm{em}}(L(T), L(\sigma))
\]
from the category of tight representations of $T$ 
to the category of effective minimal representations of $(L(T), L(\sigma))$ 
by 
\[
\mathcal{RLRS}_{(T, \sigma)}((V, \psi)):=(V\oplus \ider_{T\oplus V}(T, V), \psi\times \mathrm{id}_{\ider_{T\oplus V}(T, V)}). 
\]
\end{definition}

\begin{theorem}\label{thm_equivalence_ly}
Let $T$ be a Lie-Yamaguti algebra. 
Then the functors
\[
\mathcal{RRT}_T: \mathrm{Rep}^{\mathrm{t}}(T)\to \mathrm{Rep}^{\mathrm{em}}(L(T), T, \ider(T))
\]
and 
\[
\mathcal{RLY}_{(L(T), T, \ider(T))}: 
\mathrm{Rep}^{\mathrm{em}}(L(T), T, \ider(T))\to\mathrm{Rep}^{\mathrm{t}}(T)
\]
are inverse to each other up to natural isomorphisms, 
and hence the categories $\mathrm{Rep}^{\mathrm t}(T)$ and 
$\mathrm{Rep}^{\mathrm{em}}(L(T), T, \ider(T))$ 
are equivalent. 
\end{theorem}
\begin{proof}
For simplicity, 
we write $\mathcal{RRT}$ and $\mathcal{RLY}$ 
for $\mathcal{RRT}_T$ and $\mathcal{RLY}_{(L(T), T, \ider(T))}$. 

\smallbreak
(a)
Let us construct a natural isomorphism 
$\alpha: \mathrm{Id}_{\mathrm{Rep}^{\rm t}(T)}\overset{\sim}{\to}
\mathcal{RLY}\circ\mathcal{RRT}$ 
of functors $\mathrm{Rep}^{\rm t}(T)\to\mathrm{Rep}^{\rm t}(T)$. 

For a representation $(V, \rho, \theta, \delta)$ of $T$, 
we have 
\[
\mathcal{RLY}(\mathcal{RRT}(V))=
\mathcal{RLY}(V\oplus \ider_{T\oplus V}(T, V))
=V, 
\]
so we set $\alpha_V:=\mathrm{id}_V$. 
By (\ref{eq_tightrep_to_reptriple}), the action of $T\subseteq L(T)$ on $V\oplus \ider_{T\oplus V}(T, V)$ is given by 
\[
(x, 0)\cdot(v, f) = (\rho(x)v-f(x), -\theta_2(x, v)). 
\]
Then, 
if we write 
$\mathcal{RLY}(\mathcal{RRT}(V))=(V, \rho', \theta', \delta')$, 
then by (\ref{eq_rrt_to_rly}) we have
\begin{align*}
\rho'(x)v &= ((x, 0)\cdot(v, 0))_{\mathrm{n}}=\rho(x)v, \\
\theta'(x, y)v &= (y, 0)\cdot((x, 0)\cdot(v, 0))_{\mathrm{s}}=(y, 0)\cdot(0, -\theta_2(x, v))=(\theta_2(x, v)y, 0)
 \\
&= (\theta(x, y)v, 0). 
\end{align*}
Thus $\rho=\rho'$ and $\theta=\theta'$ holds, 
and consequently $\delta=\delta'$ also follows. 

\smallbreak
(b)
Let us construct a natural isomorphism  
$\beta: \mathrm{Id}_{\mathrm{Rep}^{\mathrm{em}}(L(T), T, \ider(T))}\overset{\sim}{\to}
\mathcal{RRT}\circ\mathcal{RLY}$ 
of functors 
$\mathrm{Rep}^{\mathrm{em}}(L(T), T, \ider(T))
\to \mathrm{Rep}^{\mathrm{em}}(L(T), T, \ider(T))$. 

Let $(M, M_{\mathrm{n}}, M_{\mathrm{s}})$ be a representation of $(L(T), T, \ider(T))$. 
By the construction (see Proposition \ref{prop_rrt_to_rly}), 
the Lie-Yamaguti algebra $T\oplus M_{\mathrm{n}}$ associated to
$M_{\mathrm{n}}=\mathcal{RLY}(M, M_{\mathrm{n}}, M_{\mathrm{s}})$ 
is the Lie-Yamaguti algebra 
$\mathcal{LY}(L(T)\oplus M, T\oplus M_{\mathrm{n}}, \ider(T)\oplus M_{\mathrm{s}})$ 
(see Proposition \ref{prop_ly_from_rt}). 
Similarly, 
by the construction of Proposition \ref{prop_rly_to_rrt}, 
the reductive triple associated to the representation 
$\mathcal{RRT}(M_{\mathrm{n}}) = M_{\mathrm{n}}\oplus\ider_{T\oplus M_{\mathrm{n}}}(T, M_{\mathrm{n}})$ 
of $(L(T), T, \ider(T))$, 
or $(L(T), T, \ider(T))\oplus \mathcal{RRT}(M_{\mathrm{n}})$, 
is the reductive triple associated to the Lie-Yamaguti algebra 
$T\oplus M_{\mathrm{n}}$, 
or $(L(T\oplus M_{\mathrm{n}}), T\oplus M_{\mathrm{n}}, \ider(T\oplus M_{\mathrm{n}}))$. 

We note that $(L(T)\oplus M, T\oplus M_{\mathrm{n}}, \ider(T)\oplus M_{\mathrm{s}})$ 
is minimal and effective. 
Then, by Proposition \ref{prop_isom_to_L} (3), 
there is a unique isomorphism 
\[
(L(T)\oplus M, T\oplus M_{\mathrm{n}}, \ider(T)\oplus M_{\mathrm{s}})
\overset{\sim}{\to}
(L(T), T, \ider(T))\oplus \mathcal{RRT}(M_{\mathrm{n}}) 
\]
of reductive triples which is the identity map on $T\oplus M_{\mathrm{n}}$. 
Furthermore, $T\oplus D(T)$ is generated by $T$ on the both sides,
and 
$M_{\mathrm{n}}+[T, M_{\mathrm{n}}]$ 
is equal to 
$M_{\mathrm{n}}\oplus M_{\mathrm{s}}$ 
on the left hand side and 
$M_{\mathrm{n}}\oplus \ider_{T\oplus M_{\mathrm{n}}}(T, M_{\mathrm{n}})$ 
on the right hand side. 
Thus we obtain an isomorphism 
\[
(M, M_{\mathrm{n}}, M_{\mathrm{s}})\overset{\sim}{\to}
\mathcal{RRT}(\mathcal{RLY}(M, M_{\mathrm{n}}, M_{\mathrm{s}}))
\]
of representations of $(L(T), T, \ider(T))$. 

It is straightforward to check the naturality. 
\end{proof}

\begin{remark}
More precisely, we have natural transformations 
\[
\beta: \mathrm{Id}_{\mathrm{Rep}^{\mathrm{m}}(L(T), T, \ider(T))} 
\to
\mathcal{RRT}\circ
\mathcal{RLY}|_{\mathrm{Rep}^{\mathrm{m}}(L(T), T, \ider(T))}
\] 
on $\mathrm{Rep}^{\mathrm{m}}(L(T), T, \ider(T))$, 
consisting of surjective homomorphisms which are bijective on the ${\mathrm{n}}$-part, 
and 
\[
\gamma: \mathcal{RRT}\circ
\mathcal{RLY}|_{\mathrm{Rep}^{\mathrm{e}}(L(T), T, \ider(T))}
\to
\mathrm{Id}_{\mathrm{Rep}^{\mathrm{e}}(L(T), T, \ider(T))} 
\]
on $\mathrm{Rep}^{\mathrm{e}}(L(T), T, \ider(T))$, 
consisting of injective homomorphisms which are bijective on the ${\mathrm{n}}$-part. 
When restricted to $\mathrm{Rep}^{\mathrm{em}}(L(T), T, \ider(T))$, 
they give natural isomorphisms. 
\end{remark}

From Corollary \ref{cor_ss_section} and Theorem \ref{thm_tightness} (1), we have the following. 
\begin{corollary}\label{cor_equivalence_ly}
If $T$ is an $L$-semisimple Lie-Yamaguti algebra satisfying $[T, T, T]=T$, 
then the categories 
$\mathrm{Rep}(T)$ and $\mathrm{Rep}^{\mathrm{em}}(L(T), T, \ider(T))$ 
are equivalent. 
\end{corollary}

\begin{remark}\label{rem_converse}
Conversely, what can we say about $T$ and $L(T)$ 
if $\mathcal{RLY}_{(L(T), T, \ider(T))}: 
\mathrm{Rep}^{\mathrm{em}}(L(T), T, \ider(T))\to\mathrm{Rep}(T)$ 
is an equivalence, or equivalently, 
if every representation of $T$ is tight? 

(1)
We see that $L(T)$ does not have to be semisimple. 
Let $T=\mathbb{K}$ with the trivial binary and ternary operations. 
Then $L(T)=T$ is an abelian $1$-dimensional Lie algebra. 

Let $(V, \rho, \theta, \delta)$ be a representation of $T$. 
Since $T$ is $1$-dimensional, 
$\theta(x, y)=\theta(y, x)$, $[\rho(x), \rho(y)]=0$ and $\rho(x*y)=0$ 
hold for any $x, y\in T$, 
and $\delta(x, y)=0$ follows from (RLY1). 
By Lemma \ref{lem_tight}, $V$ is a tight representation. 

(2)
On the other hand, we see that if 
every representation of $T$ is tight, 
then $[T, T]\cap T\subseteq [T, T, T]$ holds, 
where $[T, T]$ is the commutator in $L(T)$. 

Indeed, if $[T, T]\cap T\not\subseteq [T, T, T]$, 
then there exists a linear map $\lambda: T\to \mathbb{K}$ 
such that $\lambda([T, T, T])=0$ and $\lambda([T, T]\cap T)\not=0$. 
Let $V$ be a nonzero vector space 
and define $\rho(x)(v)=\lambda(x)v$, $\theta(x, y)(v)=0$ and $\delta(x, y)(v)=-\lambda(x*y)v$ 
(cf. Example \ref{ex_non_tight}). 
From $\lambda([T, T, T])=0$, 
we see that this gives a representation of $T$. 

An element $z\in [T, T]\cap T$ can be written as 
$z=\sum x_i*y_i$ with $\sum D_{x_i, y_i}=0$, 
and $\lambda([T, T]\cap T)\not=0$ implies the existence 
of $x_i, y_i$ with 
$\sum D_{x_i, y_i}=0$ and $\sum \delta(x_i, y_i)\not=0$. 
By Lemma \ref{lem_tight}, $V$ is not tight. 
\end{remark}

\begin{theorem}
Let $(T, \sigma)$ be an infinitesimal $s$-manifold. 
Then the functors 
\[
\mathcal{RLRS}_{(T, \sigma)}: \mathrm{Rep}^{\rm t}(T, \sigma)\to\mathrm{Rep}^{\mathrm{em}}(L(T), L(\sigma))
\]
and 
\[
\mathcal{RISM}_{(L(T), L(\sigma))}: \mathrm{Rep}^{\mathrm{em}}(L(T), L(\sigma))\to\mathrm{Rep}^{\rm t}(T, \sigma)
\] 
are inverse to each other up to natural isomorphisms, 
and hence the categories $\mathrm{Rep}^{\rm t}(T, \sigma)$ and 
$\mathrm{Rep}^{\mathrm{em}}(L(T), L(\sigma))$ 
are equivalent. 
\end{theorem}
\begin{proof}
Let us write 
$\mathcal{RLRS}$ and $\mathcal{RISM}$ 
for $\mathcal{RLRS}_{(T, \sigma)}$ and $\mathcal{RISM}_{(L(T), L(\sigma))}$. 

Let $(V, \psi)$ be an object of $\mathrm{Rep}^{\rm t}(T, \sigma)$. 
If we regard $V$ as an object of $\mathrm{Rep}^{\rm t}(T)$, 
then the isomorphism $\alpha_V$ in the proof of Theorem \ref{thm_equivalence_ly}
is actually $\mathrm{id}_V$,
and it commutes with $\psi$. 
Thus $\alpha_{(V, \psi)}=\mathrm{id}_V$ gives 
a natural isomorphism 
$\mathrm{Id}_{\mathrm{Rep}^{\rm t}(T, \sigma)}\overset{\sim}{\to}
\mathcal{RISM}\circ\mathcal{RLRS}$. 

Let $(M, M_{\mathrm{n}}, M_{\mathrm{s}}, \tilde{\psi})$ be an object of 
$\mathrm{Rep}^{\mathrm{em}}(L(T), L(\sigma))$. 
Then we can write 
$\mathcal{RLRS}(\mathcal{RISM}((M, M_{\mathrm{n}}, M_{\mathrm{s}}, \tilde{\psi})))$ 
as $(M', M_{\mathrm{n}}, M'_{\mathrm{s}}, \tilde{\psi}')$. 
The morphism $\beta_{(M, M_{\mathrm{n}}, M_{\mathrm{s}})}$ in the proof of Theorem \ref{thm_equivalence_ly}
is the identity on $M_{\mathrm{n}}$, 
and $\tilde{\psi}'|_{M_{\mathrm{n}}}=\tilde{\psi}|_{M_{\mathrm{n}}}$, 
so they commute on $M_{\mathrm{n}}$. 
On the other hand, $\tilde{\psi}|_{M_{\mathrm{s}}}$ and $\tilde{\psi}'|_{M'_{\mathrm{s}}}$ 
are identity maps, 
and therefore commute with $\beta_{(M, M_{\mathrm{n}}, M_{\mathrm{s}})}$. 
Thus $\beta_{(M, M_{\mathrm{n}}, M_{\mathrm{s}})}$ is an isomorphism 
of representations of $(L(T), L(\sigma))$. 
\end{proof}

By Corollary \ref{cor_ss_section} and Theorem \ref{thm_tightness} (2), the following holds. 
\begin{corollary}\label{cor_equivalence_ism}
Let $(T, \sigma)$ be an $L$-semisimple infinitesimal $s$-manifold. 
Then 
the categories $\mathrm{Rep}(T, \sigma)$ and 
$\mathrm{Rep}^{\mathrm{em}}(L(T), L(\sigma))$ 
are equivalent. 
\end{corollary}

\begin{remark}
We ask the same question as in Remark \ref{rem_converse}: 
what can we say about $T$ 
if every representation of $(T, \sigma)$ is tight? 

Again, $L(T)$ does not have to be semisimple. 
As in Remark \ref{rem_converse} (1), 
let $T=\mathbb{K}$ with the trivial operations 
and $\sigma: T\to T$ be given by $\sigma(x)=ax$ with $a\in\mathbb{K}$, $a\not=0, 1$. 
Then any representation of $(T, \sigma)$ is tight by the same reason. 
\end{remark}


\begin{thebibliography}{99}

\bibitem{AG2003}
N. Andruskiewitsch, M. Gra\~na, 
From racks to pointed Hopf algebras, 
Adv. Math. 178 (2003), no. 2, 177--243. 

\bibitem{BDE2005}
P. Benito, C. Draper, A. Elduque, 
Lie-Yamaguti algebras related to $\fg_2$, 
J. Pure Appl. Algebra 202 (2005), no. 1, 22--54. 

\bibitem{BD2009}
W. Bertram, M. Didry, 
Symmetric bundles and representations of Lie triple systems, 
J. Gen. Lie Theory Appl. 3 (2009), no. 4, 261--284.

\bibitem{BEMH2009}
P. Benito, A. Elduque, F. Mart\'in-Herce, 
Irreducible Lie-Yamaguti algebras, 
J. Pure Appl. Algebra 213 (2009), no. 5, 795--808. 

\bibitem{BEMH2011}
P. Benito, A. Elduque, F. Mart\'in-Herce, 
Irreducible Lie-Yamaguti algebras of generic type, 
J. Pure Appl. Algebra 215 (2011), no. 2, 108--130.

\bibitem{CJKLS2003}
J. S. Carter, D. Jelsovsky, S. Kamada, L. Langford, M. Saito, 
Quandle cohomology and state-sum invariants of knotted curves and surfaces, 
Trans. Amer. Math. Soc. 355 (2003), 3947--3989. 

\bibitem{EM2016}
M. Elhamdadi, E. M. Moutuou, 
Foundations of topological racks and quandles, 
J. Knot Theory Ramifications 25 (2016), no. 3, 1640002, 17 pp. 

\bibitem{ESZ2019}
M. Elhamdadi, M. Saito, E. Zappala, 
Continuous cohomology of topological quandles, 
J. Knot Theory Ramifications 28 (2019), no. 6, 1950036, 22 pp.

\bibitem{EGL2023}
M. Elhamdadi, T. Gona, H. Lahrani, 
Topologies, posets and finite quandles, 
Extracta Math. 38 (2023), no. 1, 1--15.

\bibitem{Fedenko1977}
A. S. Fedenko, 
``Prostranstva s simmetriyami''(Russian) (``Spaces with symmetries''), 
Izdat. Belorussk. Gos. Univ., Minsk, 1977. 

\bibitem{GN2012}
D. Gaparayi, A. Nourou Issa, 
A twisted generalization of Lie-Yamaguti algebras, 
Int. J. Algebra 6 (2012), no. 7, 339--352.

\bibitem{GMM2024}
S. Goswami, S. K. Mishra, G. Mukherjee, 
Automorphisms of extensions of Lie-Yamaguti algebras and inducibility problem, 
J. Algebra 641 (2024), 268--306. 

\bibitem{Ishikawa2018}
K. Ishikawa, 
On the classification of smooth quandles, preprint. 

\bibitem{Jackson2005}
N. Jackson, 
Extensions of racks and quandles, 
Homology Homotopy Appl. 7 (2005), no. 1, 151--167. 

\bibitem{Joyce1982}
D. Joyce, 
A classifying invariant of knots, the knot quandle, 
J. Pure Appl. Algebra 23 (1982), 37--65. 

\bibitem{Kikkawa1979}
M. Kikkawa, 
Remarks on solvability of Lie triple algebras, 
Mem. Fac. Sci. Shimane Univ. 13 (1979), 17--22.

\bibitem{Kikkawa1981}
M. Kikkawa, 
On Killing-Ricci forms of Lie triple algebras, 
Pacific J. Math. 96 (1981), 153--161. 

\bibitem{KW2001}
M. K. Kinyon, A. Weinstein, 
Leibniz algebras, Courant algebroids, and multiplications on reductive homogeneous spaces, 
Amer. J. Math. 123 (2001), 525--550.

\bibitem{Kowalski1980}
O. Kowalski, 
``Generalized Symmetric Spaces,'' 
Lecture Notes in Mathematics, 805, 
Springer-Verlag, 1980. 

\bibitem{Lister1952}
W. G. Lister, 
A structure theory of Lie triple systems, 
Trans. Amer. Math. Soc. 72 (1952), 217--242.

\bibitem{Matveev1982}
S. V. Matveev, 
Distributive groupoids in knot theory (Russian), 
Mat. Sb. 119 (1982), 78--88. 
(English translation: Math. USSR-Sb. 47 (1984), 73--83). 

\bibitem{Nomizu1954}
K. Nomizu, 
Invariant affine connections on homogeneous spaces, 
Amer. J. Math. 76 (1954), 33--65. 

\bibitem{Rubinsztein2007}
R. L. Rubinsztein, 
Topological quandles and invariants of links, 
J. Knot Theory Ramifications 16 (2007), 789--808. 

\bibitem{SZZ2021}
Y. Sheng, J. Zhao, Y. Zhou, 
Nijenhuis operators, product structures and complex structures on Lie-Yamaguti algebras, 
J. Algebra Appl. 20 (2021), no. 8, Paper No. 2150146, 22 pp.

\bibitem{SC2023}
Q. Sun, S. Chen, 
Representations and cohomologies of differential Lie-Yamaguti algebras with any weights, 
J. Lie Theory 33 (2023), no. 2, 641--662.

\bibitem{Sagle1965a}
A. A. Sagle, 
On anti-commutative algebras and general Lie triple systems, 
Pacific J. Math. 15 (1965), 281--291. 

\bibitem{Sagle1965b}
A. A. Sagle, 
On simple algebras obtained from homogeneous general Lie triple systems, 
Pacific J. Math. 15 (1965), 1397--1400. 

\bibitem{Sagle1968}
A. A. Sagle, 
A note on simple anti-commutative algebras obtained from reductive homogeneous spaces, 
Nagoya Math. J. 31 (1968), 105--124. 

\bibitem{Takahashi2016}
N. Takahashi, 
Quandle varieties, generalized symmetric spaces, and $\varphi$-spaces, 
Transform. Groups 21 (2016), 555--576. 

\bibitem{Takahashi2021}
N. Takahashi, 
Modules over geometric quandles and representations of Lie-Yamaguti algebras, 
J. Lie Theory 31 (2021), no. 4, 897--932. 

\bibitem{TJL2023}
W. Teng, J. Jin, F. Long, 
Relative Rota-Baxter operators on Hom-Lie-Yamaguti algebras, 
J. Math. Res. Appl. 43 (2023), no. 6, 648--664.

\bibitem{Uematsu2023}
K. Uematsu, 
On classification of irreducible quandle modules over a connected quandle, 
Hiroshima Math. J. 53 (2023), no. 1, 27--61.

\bibitem{Yamaguti1958}
K. Yamaguti, 
On the Lie triple system and its generalization, 
J. Sci. Hiroshima Univ. Ser. A, 21 (1957/1958), 155--160.

\bibitem{Yamaguti1969}
K. Yamaguti, 
On cohomology groups of general Lie triple systems, 
Kumamoto J. Sci., Ser. A, 8 (1969), no. 4, 135--146. 

\bibitem{ZL2015}
T. Zhang, J. Li, 
Deformations and extensions of Lie-Yamaguti algebras, 
Linear Multilinear Algebra 63 (2015), no. 11, 2212--2231.

\bibitem{ZL2017}
T. Zhang, J. Li, 
Representations and cohomologies of Hom-Lie-Yamaguti algebras with applications, 
Colloq. Math. 148 (2017), no. 1, 131--155.

\bibitem{ZXQ2023}
J. Zhao, S. Xu, Y. Qiao, 
Post Lie-Yamaguti algebras, relative Rota-Baxter operators of nonzero weights, and their deformations, 
math.arXiv:2304.06324. 


\end{thebibliography}
\end{document}